\documentclass[11pt]{amsart}
\usepackage{amsbsy}
\usepackage{graphicx,epsfig,subfigure,psfrag}
\usepackage{color}
\usepackage{amsmath}%
\usepackage{amsfonts}%
\usepackage{amssymb}%
\begin{document}

%%%%%%%%%%%%%%%%%%%%%%%%%%%%%%%%%%%%%%%%%%%%%%%%%%%%%%%%%%%%%%%%%%%%
% Theorem, definition, lemma, proposition, corollary and proof
%%%%%%%%%%%%%%%%%%%%%%%%%%%%%%%%%%%%%%%%%%%%%%%%%%%%%%%%%%%%%%%%%%%%
%%%%%%%%%%%%%%%%%%%%%%%%%%%%%%%%%%%%%%%%%%%%%%%%
\newtheorem{theorem}{Theorem}%[section]
\newtheorem{proposition}{Proposition}%[section]
\newtheorem{lemma}{Lemma}%[section]
\newtheorem{corollary}{Corollary}%[section]%%
\newtheorem{definition}{Definition}%[section]
\newtheorem{remark}{Remark}%[section]
\newtheorem{remarks}{Remarks}%[section]
%%%%%%%%%%%%%%%%%%%%%%%%%%%%%%%%%%%
%%%%%%%%%%%%%%%%%%%%%%%%%%%%%%%%%%%%%%%%%%%%%%
\newcommand{\tex}{\textstyle}
%%%%%%%%%%%%%%%%%%%%%%%%%%%%%%%%%%%%%%%%%%%%
%%%%%%%%%%%%%%%%%%%%%%%%%%%%%%%%%%%%%%%%%%%%
\numberwithin{equation}{section} \numberwithin{theorem}{section}
\numberwithin{proposition}{section} \numberwithin{lemma}{section}
\numberwithin{corollary}{section}
\numberwithin{definition}{section} \numberwithin{remark}{section}
%%%%%%%%%%%%%%%%%%%%%%%%%%%%%%%%%%%%%%%%%%%%
%%%%%%%%%%%%%%%%%%%%%%%%%%%%%%%%%%%%%%%%%%%%
\newcommand{\ren}{\mathbb{R}^N}
\newcommand{\re}{\mathbb{R}}
\newcommand{\dyle}{\displaystyle}
\newcommand{\n}{\nabla}
\newcommand{\p}{\partial}
\newcommand{\iy}{\infty}
\newcommand{\pa}{\partial}
\newcommand{\fp}{\noindent}
\newcommand{\ms}{\medskip\vskip-.1cm}
\newcommand{\mpb}{\medskip}
%%%%%%%%%%%%%%%%%%%%%%%%%%%%%%%%%%%%%%%%%%%%%%%%%
\newcommand{\AAA}{{\bf A}}
\newcommand{\BB}{{\bf B}}
\newcommand{\CC}{{\bf C}}
\newcommand{\DD}{{\bf D}}
\newcommand{\EE}{{\bf E}}
\newcommand{\FF}{{\bf F}}
\newcommand{\GG}{{\bf G}}
\newcommand{\oo}{{\mathbf \omega}}
\newcommand{\Am}{{\bf A}_{2m}}
\newcommand{\CCC}{{\mathbf  C}}
\newcommand{\II}{{\mathrm{Im}}\,}
\newcommand{\RR}{{\mathrm{Re}}\,}
\newcommand{\eee}{{\mathrm  e}}

\newcommand{\cN}{{\mathcal{N}}}

%%%%%%%%%%%%%%%%%%%%%%%%%%%%%%%%%%%%%%%%%%%%%%%%%%%%%%%%%%%%%%%%%%%%%%% L^2\rho...
\newcommand{\LL}{L^2_\rho(\ren)}
\newcommand{\LLL}{L^2_{\rho^*}(\ren)}
%%%%%%%%%%%%%%%%%%%%%%%%%%%%%%%%%%
%%%%%%%%%%%%%%%%%%%%%%%%%%%%%%%%%%%%%%%%%%%%%%%%%%%%
\renewcommand{\a}{\alpha}
\renewcommand{\b}{\beta}
\newcommand{\g}{\gamma}
\newcommand{\G}{\Gamma}
\renewcommand{\d}{\delta}
\newcommand{\D}{\Delta}
\newcommand{\e}{\varepsilon}
\newcommand{\var}{\varphi}
\renewcommand{\l}{\lambda}
\renewcommand{\o}{\omega}
\renewcommand{\O}{\Omega}
\newcommand{\s}{\sigma}
\renewcommand{\t}{\tau}
\renewcommand{\th}{\theta}
\newcommand{\z}{\zeta}
\newcommand{\wx}{\widetilde x}
\newcommand{\wt}{\widetilde t}
\newcommand{\noi}{\noindent}
 %%%%%%%%%%%%%%%%%%%%%%%%%%%%%%%%%%%%%%%%%%%
\newcommand{\uu}{{\bf u}}
\newcommand{\xx}{{\bf x}}
\newcommand{\yy}{{\bf y}}
\newcommand{\zz}{{\bf z}}
\newcommand{\aaa}{{\bf a}}
\newcommand{\cc}{{\bf c}}
\newcommand{\jj}{{\bf j}}
\newcommand{\UU}{{\bf U}}
\newcommand{\YY}{{\bf Y}}
\newcommand{\HH}{{\bf H}}
\newcommand{\GGG}{{\bf G}}
\newcommand{\VV}{{\bf V}}
\newcommand{\ww}{{\bf w}}
\newcommand{\vv}{{\bf v}}
\newcommand{\hh}{{\bf h}}
\newcommand{\di}{{\rm div}\,}
\newcommand{\ii}{{\rm i}\,}
\def\I{{\rm Id}}
%%%%%%%%%%%%%%%%%%%%%%%%%%%%%%%%%%
%%%%%%%%%%%%%%%%%%%%%%%%%%%%%%%%%%%%%   VAG, NEW
\newcommand{\inA}{\quad \mbox{in} \quad \ren \times \re_+}
\newcommand{\inB}{\quad \mbox{in} \quad}
\newcommand{\inC}{\quad \mbox{in} \quad \re \times \re_+}
\newcommand{\inD}{\quad \mbox{in} \quad \re}
\newcommand{\forA}{\quad \mbox{for} \quad}
\newcommand{\whereA}{,\quad \mbox{where} \quad}
\newcommand{\asA}{\quad \mbox{as} \quad}
\newcommand{\andA}{\quad \mbox{and} \quad}
\newcommand{\withA}{,\quad \mbox{with} \quad}
\newcommand{\orA}{,\quad \mbox{or} \quad}
\newcommand{\atA}{\quad \mbox{at} \quad}
\newcommand{\onA}{\quad \mbox{on} \quad}
\newcommand{\ef}{\eqref}
\newcommand{\mc}{\mathcal}
\newcommand{\mf}{\mathfrak}

\newcommand{\Ge}{\G_\e}
\newcommand{\Hn}{ H^{1}(\Rn)}
\newcommand{\Wn}{W^{1,2}(\Rn)}
\newcommand{\Wan}{W^{\frac{\a}{2},2}(\Rn)}
\newcommand{\Wa}{W^{\frac{\a}{2},2}(\R)}
\newcommand{\intn}{\int_{\Rn}}
\newcommand{\intR}{\int_\R}
\newcommand{\ie}{I_\e}
\newcommand{\nie}{\n \ie}
\newcommand{\gie}{I_\e'}
\newcommand{\ies}{I_\e''}
\newcommand{\ios}{I_0''}
\newcommand{\ip}{I'_0}

\newcommand{\zex}{z_{\e,\rho}}
\newcommand{\wex}{w_{\e,\xi}}
\newcommand{\zer}{z_{\e,\rho}}
\newcommand{\wer}{w_{\e,\rho}}
\newcommand{\dzex}{{\dot{z}}_{\e,\rho}}
\newcommand{\bE}{{\bf E}}
\newcommand{\bu}{{\bf u}}
\newcommand{\bv}{{\bf v}}
\newcommand{\bz}{{\bf z}}
\newcommand{\bw}{{\bf w}}
\newcommand{\bo}{{\bf 0}}
\newcommand{\bp}{{\bf \phi}}
\newcommand{\up}{\underline{\phi}}
\newcommand{\bh}{{\bf h}}

%%%%%%%%%%%%%%%%%%%%%%%%%%

\newcommand{\E}{\mathbb{E}}
\newcommand{\X}{\mathbb{X}}
\newcommand{\F}{\mathbb{F}}
\newcommand{\Y}{\mathbb{Y}}
\newcommand{\M}{\mathbb{M}}
\newcommand{\h}{\mathbb{H}}
%%%%%%%%%%%%%%%%%%%%%%%%%%%%%%%
\newcommand{\ssk}{\smallskip}
\newcommand{\LongA}{\quad \Longrightarrow \quad}
%%%%%%%%%%%%%%%%%%%%%%%%%%%%%%%%
%%%%%%%%%%%%%%%%%%%%%%%%%%%%%%%%%%
\def\com#1{\fbox{\parbox{6in}{\texttt{#1}}}}
%%%%%%%%%%%%%%%%%%%%%%%%%%%%%%%%%%
%%%%%%%%%%%%%%%%%%% From Paper1
\def\N{{\mathbb N}}
\def\A{{\cal A}}
\newcommand{\de}{\,d}
\newcommand{\eps}{\varepsilon}
\newcommand{\be}{\begin{equation}}
\newcommand{\ee}{\end{equation}}
\newcommand{\spt}{{\mbox spt}}
\newcommand{\ind}{{\mbox ind}}
\newcommand{\supp}{{\mbox supp}}
\newcommand{\dip}{\displaystyle}
\newcommand{\prt}{\partial}
\renewcommand{\theequation}{\thesection.\arabic{equation}}
\renewcommand{\baselinestretch}{1.1}
%%%%%%%%%%%%%%%%%%%%%%%%%%%%%%%%%%%%%%%%%%%%%%%
\newcommand{\Dm}{(-\D)^m}

\newenvironment{pf}{\noindent{\it
Proof}.\enspace}{\rule{2mm}{2mm}\medskip}

\newcommand{\lapa}{(-\Delta)^{\alpha/2}}

%%%%%%%%%%%%%%%%%%%%%%%%%
\title
%%%%%
%%%%%%%%%%%%%%%%%%%%%%%%%
%%%%
{\bf Positive solutions for semilinear fractional elliptic problems involving an inverse fractional operator}

\author{P.~\'Alvarez-Caudevilla, E.~Colorado and Alejandro Ortega}

\address{Departamento de Matem\'aticas, Universidad Carlos III de Madrid,
Av. Universidad 30,
28911 Legan\'es (Madrid), Spain}
\email{pacaudev@math.uc3m.es}

\address{Departamento de Matem\'aticas, Universidad Carlos III de Madrid,
Av. Universidad 30,
28911 Legan\'es (Madrid), Spain}
\email{ecolorad@math.uc3m.es}

\address{Departamento de Matem\'aticas, Universidad Carlos III de Madrid,
Av. Universidad 30,
28911 Legan\'es (Madrid), Spain}
\email{alortega@math.uc3m.es}

\thanks{This paper has been partially supported by the Ministry of Economy and Competitiveness of
Spain and FEDER, under research project MTM2016-80618-P}

\thanks{The first author was also partially supported by the Ministry of Economy and Competitiveness of
Spain under research project RYC-2014-15284}

\date{\today}

%%%%%%%%%%%%%%%%%%%%%%%%%%%

\begin{abstract}
This paper is devoted to the study of the existence of positive solutions for a problem related to a higher order fractional differential equation involving a nonlinear term depending on a fractional differential operator,
\begin{equation*}
\left\{
\begin{tabular}{lcl}
$(-\Delta)^{\alpha} u=\lambda u+ (-\Delta)^{\beta}|u|^{p-1}u$ & &in $\Omega$, \\
  $\mkern+3mu(-\Delta)^{j}u=0$                                        & &on $\partial\Omega$, for $j\in\mathbb{Z}$, $0\leq j< [\alpha]$,
\end{tabular}
\right.
\end{equation*}
where $\Omega$ is a bounded domain in $\mathbb{R}^{N}$, $0<\beta<1$,
$\beta<\alpha<\beta+1$ and $\lambda>0$. In particular, we study the fractional elliptic problem,
\begin{equation*}
        \left\{
        \begin{array}{ll}
        (-\Delta)^{\alpha-\beta} u= \lambda(-\Delta)^{-\beta}u+ |u|^{p-1}u  & \hbox{in} \quad \Omega, \\
        \mkern+72.2mu  u=0 & \hbox{on} \quad \partial\Omega,
        \end{array}
        \right.
\end{equation*}
and we prove existence or nonexistence of positive solutions depending on the parameter $\lambda>0$,
up to the critical value of the exponent $p$, i.e., for $1<p\leq 2_{\mu}^*-1$ where $\mu:=\alpha-\beta$ and $2_{\mu}^*=\frac{2N}{N-2\mu}$ is the critical
exponent of the Sobolev embedding.
\end{abstract}
\maketitle
\noindent {\it \footnotesize 2010 Mathematics Subject Classification}. {\scriptsize 35A15, 35G20, 35J61, 49J35.}\\
{\it \footnotesize Key words}. {\scriptsize Fractional Laplacian, Critical Problem, Concentration-Compactness Principle, Mountain Pass Theorem}

%%%%%%%%%%%%%%%%%%%%%%%%%%%%%%%%%%%%%%%%%%%%%%%%%%%%%%%%%%%%%%%%%%%%
%%%%%%%%%%%%%%%%%%%%%%%%%%%%%%%%%%%%%%%%%%%%%%%%%%%%%%%%%%%%%%%%%%%%
%%%%%%%%%%%%%%%%%%%%%%%%%%%%%%%%%%%%%%%%%%%%%%%%%%%%%%%%%%%%%%%%%%%%
\section{Introduction}\label{sec:intro}
%%%%%%%%%%%%%%%%%%%%%%%%%%%%%%%%%%%%%%%%%%%%%%%%%%%%%%%%%%%%%%%%%%%%
%%%%%%%%%%%%%%%%%%%%%%%%%%%%%%%%%%%%%%%%%%%%%%%%%%%%%%%%%%%%%%%%%%%%
%%%%%%%%%%%%%%%%%%%%%%%%%%%%%%%%%%%%%%%%%%%%%%%%%%%%%%%%%%%%%%%%%%%%
\noindent  Let $\O$ be a smooth bounded domain of $\ren$ with $N>2\mu$ and
$$\mu:=\alpha-\beta\quad \hbox{with} \quad 0<\beta<1\quad \hbox{and}\quad \beta<\alpha<\beta+1.$$
We analyze the existence of positive solutions for the following fractional elliptic problem,
\begin{equation}\label{ecuacion}
         \left\{
        \begin{array}{ll}
        (-\Delta)^{\alpha-\beta} u= \g(-\Delta)^{-\beta}u+ |u|^{p-1}u  & \hbox{in} \quad \Omega, \\
    u= 0 & \hbox{on} \quad \partial\Omega,
        \end{array}
        \right.
                \tag{$P_\gamma$}
\end{equation}
depending on the real parameter $\gamma>0$. To this end, we consider,
\begin{equation*}
1<p\leq2_{\mu}^*-1=\frac{N+2\mu}{N-2\mu},
\end{equation*}
where $2_{\mu}^*=\frac{2N}{N-2\mu}$ is the critical exponent of the Sobolev embedding. Associated with \eqref{ecuacion} we have the following Euler--Lagrange functional:
\begin{equation}
\label{funcional_ecuacion}
    \mc{F}_\g(u)=\frac{1}{2}\int_\O|(-\Delta)^{\frac{\mu}{2}} u|^2dx-\frac{\g}{2} \int_\O |(-\D)^{-\frac{\beta}{2}}u|^2\, dx-\frac{1}{p+1} \int_\O |u|^{p+1}dx,
\end{equation}
 such that the solutions of \eqref{ecuacion} corresponds to critical  points of the $C^1$ functional \eqref{funcional_ecuacion} and vice versa.

Note that $(-\D)^{-\beta}$ is  a positive linear integral compact operator from $L^2(\O)$ into itself and it is well defined thanks to the Spectral Theorem.
The definition of the fractional powers of the positive Laplace operator $(-\Delta)$, in a bounded domain $\Omega$ with
homogeneous Dirichlet boundary data, can be carried out through the spectral decomposition using the powers of the eigenvalues of $(-\Delta)$
with the same boundary conditions. Indeed, let $(\varphi_i,\lambda_i)$ be the eigenfunctions (normalized with respect to the $L^2(\O)$-norm) and eigenvalues of $(-\Delta)$ under homogeneous
Dirichlet boundary data. Then, $(\varphi_i,\lambda_i^{\mu})$ stand for the eigenpairs of $(-\Delta)^{\mu}$ under homogeneous
Dirichlet boundary conditions as well. Thus, the fractional operator $(-\Delta)^{\mu}$ is well defined in the space of functions that vanish on the boundary,
\begin{equation*}
H_0^{\mu}(\Omega)=\left\{u=\sum_{j=1}^{\infty} a_j\varphi_j\in L^2(\Omega):\ \|u\|_{H_0^{\mu}(\Omega)}=\left(\sum_{j=1}^{\infty} a_j^2\lambda_j^{\mu} \right)^{\frac{1}{2}}<\infty\right\}.
\end{equation*}
As a result of this definition it follows that,
\begin{equation}\label{eqnorma}
\|u\|_{H_0^{\mu}(\O)}=\|(-\Delta)^{\frac{\mu}{2}}u\|_{L^2(\O)}.
\end{equation}
In particular,
\begin{equation*}
(-\Delta)^{-\beta}u=\sum_{j=1}^{\infty} a_j\lambda_j^{-\beta}\varphi_j.
\end{equation*}
Since the above definition allows us to integrate by parts, we say that $u\in H_0^{\mu}(\O)$ is an energy or weak solution for problem \eqref{ecuacion} if,
\begin{equation*}
\int_{\Omega}(-\Delta)^{\frac{\mu}{2}}u(-\Delta)^{\frac{\mu}{2}}\phi dx=\gamma\int_{\O}(-\Delta)^{-\frac{\beta}{2}}u(-\Delta)^{-\frac{\beta}{2}}\phi dx+\int_{\O}|u|^{p-1}u\phi dx,\quad \forall\phi\in H_0^{\mu}(\O).
\end{equation*}
In other words, $u\in H_0^{\mu}(\O)$ is a critical point of the functional defined by \eqref{funcional_ecuacion}. We also observe that
the functional embedding features for the equation in \eqref{ecuacion} are governed by the  Sobolev's embedding Theorem. Let us recall the compact inclusion,
\begin{equation}
\label{compact_emb}
H_0^{\mu}(\O)  \hookrightarrow L^{p+1}(\O),\quad 2\leq p+1<2_{\mu}^*,
\end{equation}
being a continuous inclusion up to the critical exponent $p=2_{\mu}^*-1$.\newline
To define non-integer higher-order powers for the Laplace operator, let us recall that the homogeneous Navier boundary conditions are defined as
\begin{equation*}
u=\Delta u=\Delta^2 u=\ldots=\Delta^{k-1} u=0,\quad\mbox{on }\partial\Omega.
\end{equation*}
Given $\alpha>1$, the $\alpha$-th power of the classical Dirichlet Laplacian in the sense of the spectral theory can be defined as the operator
whose action on a smooth function $u$ satisfying the homogenous Navier boundary conditions for $0\leq k<[\alpha]$ (where $[\cdot]$ means the integer part), is given by
\begin{equation*}
\langle (-\Delta)^{\alpha} u, u \rangle=\sum_{j\ge 1}
\lambda_j^{\alpha}|\langle u_1,\varphi_j\rangle|^2.
\end{equation*}
We refer to \cite{MusNa2,MusNa3} for a study of this higher-order fractional Laplace operator, referred to as the Navier fractional Laplacian, as well as useful properties of the fractional Sobolev space $H_0^{\alpha}(\Omega)$.\newline
On the other hand, we have a connection between problem \eqref{ecuacion} and a fractional order
elliptic system which turns out to be very useful in the sequel. In particular, taking $\tex{\psi:=(-\D)^{-\beta}u}$, problem \eqref{ecuacion} provides us with the fractional elliptic  cooperative system,
\begin{equation}
\label{cosys}
    \left\{\begin{array}{l}
    (-\Delta)^{\mu}u = \gamma \psi+|u|^{p-1}u,\\
    (-\D)^{\beta}\psi=u,
    \end{array}\right.\quad \hbox{in}\quad \O,\quad (u,\psi)=(0,0)\quad\hbox{in}\quad \partial\O.
\end{equation}
Nevertheless, system \eqref{cosys} is not a variational system. In order to obtain a variational system from problem \eqref{ecuacion} we follow a similar idea to the one performed above,
distinguishing whether $\alpha=2\beta$ or $\alpha\neq2\beta$. In the first case we take $\tex{v:=\sqrt{\gamma}\psi}$ and, recalling that $\mu:=\alpha-\beta$, we obtain the following fractional elliptic cooperative system,
\begin{equation}
\label{sistemabb}
\left\{\begin{array}{l}
(-\Delta)^{\beta}u=\sqrt{\gamma}v+|u|^{p-1}u,\\
(-\Delta)^{\beta}v=\sqrt{\gamma}u,
\end{array}
\right.
\tag{$S_{\gamma}^{\beta}$}
\quad \hbox{in}\quad \O,\quad (u,v)=(0,0)\quad\hbox{on}\quad \partial\O,
\end{equation}
whose associated energy functional is
\begin{equation*}
    \mc{J}_{\g}^{\beta}(u,v)=\frac{1}{2} \int_\O |(-\D)^{\frac{\beta}{2}} u|^2dx + \frac{1}{2} \int_\O |(-\D)^{\frac{\beta}{2}} v|^2dx -\sqrt{\gamma}\int_\O uvdx -\frac{1}{p+1} \int_\O |u|^{p+1}dx.
\end{equation*}
In the second case, $\alpha\neq2\beta$, taking $v=\gamma^{\beta/\alpha}\psi$ we obtain the system,
\begin{equation*}
\left\{\begin{array}{rl}
(-\Delta)^{\mu}u=&\!\!\!\gamma^{1-\beta/\alpha}v+|u|^{p-1}u,\\
(-\Delta)^{\beta}v=&\!\!\!\gamma^{\beta/\alpha}u,
\end{array}
\right.
\quad \hbox{in}\quad \O,\quad (u,v)=(0,0)\quad\hbox{on}\quad \partial\O.
\end{equation*}
Since the former system is still not variational, we transform it into the following variational system,
\begin{equation}
\label{sistemaab}
\left\{\begin{array}{rl}
\frac{1}{\gamma^{1-\beta/\alpha}}(-\Delta)^{\mu} u =&\!\!\! v+\frac{1}{\gamma^{1-\beta/\alpha}}|u|^{p-1}u,\\
\frac{1}{\gamma^{\beta/\alpha}}(-\Delta)^{\beta}v=&\!\!\! u,
\end{array}
\right.
\tag{$S_{\gamma}^{\alpha,\beta}$}
\quad \hbox{in}\quad \O,\quad (u,v)=(0,0)\quad\hbox{on}\quad \partial\O.
\end{equation}
whose associated functional is
\begin{align*}
\mc{J}_{\g}^{\alpha,\beta}(u,v)=&\frac{1}{2\gamma^{1-\beta/\alpha}}\int_\O|(-\D)^{\frac{\mu}{2}}u|^2dx +\frac{1}{2\gamma^{\beta/\alpha}}\int_\O|(-\D)^{\frac{\beta}{2}}v|^2dx-\int_\O uv dx\\
&-\frac{1}{(p+1)\gamma^{1-\beta/\alpha}} \int_\O |u|^{p+1}dx.
\end{align*}
We will use the equivalence between problem \eqref{ecuacion} and systems
\eqref{sistemabb} and \eqref{sistemaab} to surpass the difficulties that arise while working with the inverse fractional Laplace operator $(-\Delta)^{-\beta}$. In particular, this approach will help us to avoid ascertaining explicit
estimations for this inverse term. On the other hand, to overcome the usual difficulties that appear when dealing with fractional
Laplace operators we will use the ideas of Caffarelli and Silvestre \cite{CS}, together
with those performed in \cite{BrCdPS}, giving an equivalent definition of the fractional operator $(-\Delta)^{\mu}$
in a bounded domain $\Omega$ by means of an auxiliary problem that we will introduce below.
Associated with the domain $\Omega$ let us consider the cylinder
$\mathcal{C}_{\Omega}=\Omega\times(0,\infty)\subset\mathbb{R}_+^{N+1}$ called extension cylinder.
Moreover, we denote by $(x,y)$ the points belonging to $\mathcal{C}_{\Omega}$ and with
$\partial_L\mathcal{C}_{\Omega}=\partial\Omega\times(0,\infty)$ the lateral boundary
of the extension cylinder. Thus, given a function $u\in H_{0}^{\mu}(\Omega)$, define the
$\mu$-harmonic extension function $w$, denoted by $w:=E_{\mu}[u]$, as the solution to problem,
\begin{equation*}
        \left\{
        \begin{array}{ll}
        -{\rm div}(y^{1-2\mu}\nabla w)=0 & \hbox{in} \quad \mathcal{C}_{\Omega}, \\
        w=0 & \hbox{on}\quad \partial_L\mathcal{C}_{\Omega}, \\
        w(x,0)=u(x) & \hbox{in} \quad \Omega\times\{y=0\}.
        \end{array}
        \right.
\end{equation*}
This extension function $w$ belongs to the space
\begin{equation*}
\mathcal{X}_0^{\mu}(\mathcal{C}_{\Omega})=\overline{\mathcal{C}_0^{\infty}(\Omega\times[0,\infty))}^{\|\cdot\|_{\mathcal{X}_0^{\mu}(\mathcal{C}_{\Omega})}},\ \text{with}\ \|w\|_{\mathcal{X}_0^{\mu}(\mathcal{C}_{\Omega})}^2=\kappa_{\mu}\int_{\mathcal{C}_{\Omega}}y^{1-2\mu}|\nabla w(x,y)|^2dxdy.
\end{equation*}
With that constant  $\kappa_{\mu}$, whose precise value can be seen in \cite{BrCdPS}, the extension operator is an isometry between $H_0^{\mu}(\Omega)$ and
$\mathcal{X}_0^{\mu}(\mathcal{C}_{\Omega})$ in the sense
\begin{equation}\label{isometry}
\|E_{\mu}[\varphi]\|_{\mathcal{X}_0^{\mu}(\mathcal{C}_{\Omega})}=\|\varphi\|_{H_0^{\mu}(\Omega)},\ \text{for all}\ \varphi\in H_0^{\mu}(\Omega).
\end{equation}
The relevance of the extension function $w$ is that it is related to the fractional
Laplacian of the original function through the formula
\begin{equation*}
\frac{\partial w}{\partial \nu^{\mu}}:= -\kappa_{\mu} \lim_{y\to
0^+} y^{1-2\mu}\frac{\partial w}{\partial y}=(-\Delta)^{\mu}u(x).
\end{equation*}
In the case $\Omega=\mathbb{R}^N$ this formulation provides us with explicit expressions for both the fractional Laplacian and the $\mu$-extension in terms of the Riesz and the Poisson kernels respectively. Precisely,
\begin{equation*}
\begin{split}
 & (-\Delta)^{\mu}u(x)=\ d_{N,\mu}P.V.\int_{\mathbb{R}^N}\frac{u(x)-u(y)}{|x-y|^{N+2\mu}}dy\\
& w(x,y)=\ P_y^{\mu}\ast u(x)=c_{N,\mu}y^{2\mu}\int_{\mathbb{R}^N}\frac{u(z)}{(|x-z|^2+y^2)^{\frac{N+2\mu}{2}}}dz.
\end{split}
\end{equation*}
For exact values of the constants $c_{N,\mu}$ and $d_{N,\mu}$ we refer to \cite{BrCdPS}.
Thanks to the arguments shown above, we can reformulate problem \eqref{ecuacion} in terms of the extension problem as follows,
\begin{equation}\label{extension_problem}
        \left\{
        \begin{array}{ll}
        -{\rm div}(y^{1-2\mu}\nabla w)=0  & \hbox{in}\quad \mathcal{C}_{\Omega}, \\
        w=0   & \hbox{on}\quad \partial_L\mathcal{C}_{\Omega}, \\
        \frac{\partial w}{\partial \nu^\mu}=\gamma (-\Delta)^{-\beta}w+|w|^{p-1}w & \hbox{in}\quad \Omega\times\{y=0\}.
        \end{array}
        \right.
        \tag{$\tilde{P}_{\gamma}$}
\end{equation}
Therefore, an energy or weak solution of this problem is a function $w\in \mathcal{X}_0^{\mu}(\mathcal{C}_{\Omega})$ satisfying
\begin{equation*}
\kappa_{\mu}\int_{\mathcal{C}_{\Omega}} y^{1-2\mu}\langle\nabla w,\nabla\varphi \rangle dxdy=\int_{\Omega} \left(\gamma (-\Delta)^{-\beta}w+|w|^{p-1}\o\right)\varphi(x,0)dx,\quad \forall\varphi\in\mathcal{X}_0^{\mu}(\mathcal{C}_{\Omega}).
\end{equation*}
For any energy solution $w\in \mathcal{X}_0^{\mu}(\mathcal{C}_{\Omega})$
to problem \eqref{extension_problem}, the corresponding trace function $u=Tr[w]=w(\cdot,0)$ belongs to the space $H_0^{\mu}(\Omega)$ and is an energy solution for the problem
\eqref{ecuacion} and vice versa. If $u\in H_0^{\mu}(\Omega)$ is an energy solution of \eqref{ecuacion}, then $w:=E_\mu[u]\in \mathcal{X}_0^{\mu}(\mathcal{C}_{\Omega})$ is
an energy solution for \eqref{extension_problem} and, as a consequence, both formulations are equivalent. Finally, the energy functional associated with problem \eqref{extension_problem} is
\begin{equation*}
\widetilde{\mc{F}}_{\gamma}(w)=\frac{\kappa_{\mu}}{2}\int_{\mathcal{C}_{\Omega}}y^{1-2\mu}|\nabla w|^2dxdy-\frac{\gamma}{2}\int_{\Omega}|(-\Delta)^{-\frac{\beta}{2}}w|^2dx-\frac{1}{p+1}\int_{\Omega}|w|^{p+1}dx.
\end{equation*}
Since the extension operator is an isometry, critical points of $\widetilde{\mc{F}}_{\gamma}$ in $\mathcal{X}_0^{\mu}(\mathcal{C}_{\Omega})$ correspond to critical
points of the functional $\mc{F}_{\gamma}$ in $H_0^{\mu}(\Omega)$. Indeed, arguing as in \cite[Proposition 3.1]{BCdPS},
the minima of $\widetilde{\mc{F}}_{\gamma}$ also correspond to the minima of the functional $\mc{F}_{\gamma}$.

Another useful tool to be applied throughout this work will be the following trace inequality,
\begin{equation}\label{sobext}
\int_{\mathcal{C}_{\Omega}}y^{1-2\mu}|\nabla \phi(x,y)|^2dxdy\geq C\left(\int_{\Omega}|\phi(x,0)|^rdx\right)^{\frac{2}{r}},\quad\forall\phi\in \mathcal{X}_{0}^{\mu}(\mathcal{C}_{\Omega}),
\end{equation}
with $1\leq r\leq\frac{2N}{N-2\mu},\ N>2\mu$. Let us notice that, since the extension operator is an isometry, inequality \eqref{sobext} is equivalent to the fractional Sobolev inequality,
\begin{equation}\label{sobolev}
\int_{\Omega}|(-\Delta)^{\mu/2}\varphi|^2dx\geq C\left(\int_{\Omega}|\varphi|^rdx\right)^{\frac{2}{r}},\quad\forall\varphi\in H_{0}^{\mu}(\Omega),
\end{equation}
with $1\leq r\leq\frac{2N}{N-2\mu}$, $N>2\mu$.
\begin{remark} When $r=2_{\mu}^*$, the best constant in \eqref{sobext} will be denoted by $S(\mu,N)$. This constant is explicit and independent of the domain $\Omega$. Indeed,
its exact value is given by the expression
\begin{equation*}
S(\mu,N)=\frac{2\pi^\mu\Gamma(1-\mu)\Gamma(\frac{N+2\mu}{2})(\Gamma(\frac{N}{2}))^{\frac{2\mu}{N}}}{\Gamma(\mu)\Gamma(\frac{N-2\mu}{2})(\Gamma(N))^\mu},
\end{equation*}
and it is never achieved when $\Omega$ is a bounded domain. Thus, we have,
\begin{equation*}
\int_{\mathbb{R}_{+}^{N+1}}\!\!y^{1-2\mu}|\nabla \phi(x,y)|^2dxdy\geq S(\mu,N)\left(\int_{\mathbb{R}^{N}}|\phi(x,0)|^{\frac{2N}{N-2\mu}}dx\right)^{\frac{N-2\mu}{N}}\  \forall \phi\in \mathcal{X}_0^\mu(\mathbb{R}_{+}^{N+1}).
\end{equation*}
If $\Omega=\mathbb{R}^N$, the constant $S(\mu,N)$ is achieved for the family of extremal functions $w_{\varepsilon}^{\mu}= E_\mu[v_{\varepsilon}^{\mu}]$ with
\begin{equation}\label{u_eps}
v_{\varepsilon}^{\mu}(x)=\frac{\varepsilon^{\frac{N-2\mu}{2}}}{(\varepsilon^2+|x|^2)^{\frac{N-2\mu}{2}}},
\end{equation}
for arbitrary $\varepsilon>0$; see \cite{BrCdPS} for further details. Finally, combining the previous comments, the
best constant in \eqref{sobolev} with $\Omega=\mathbb{R}^N$ is given then by $\kappa_\mu S(\mu,N)$.
\end{remark}
Although systems \eqref{sistemabb} and \eqref{sistemaab} no longer contain an inverse term as $(-\D)^{-\beta}$ they still are non-local systems, with all the complications that this entails.
However, we use the extension technique shown above to reformulate the non-local systems \eqref{sistemabb} and \eqref{sistemaab} in terms of the following local systems.
Taking $w:=E_{\mu}[u]$ and $z:=E_{\beta}[v]$, the extension system corresponding to \eqref{sistemabb} reads as
\begin{equation}\label{extension_systembb}
        \left\{
        \begin{array}{ll}
        -{\rm div}(y^{1-2\beta}\nabla w)= 0 & \hbox{in}\quad  \mathcal{C}_{\Omega}, \\
        -{\rm div}(y^{1-2\beta}\nabla z)=0 & \hbox{in}\quad \mathcal{C}_{\Omega}, \\
        \displaystyle\frac{\partial w}{\partial \nu^{\beta}}= \sqrt{\gamma} z+|w|^{p-1}w & \hbox{in}\quad \Omega\times\{y=0\},\\
        \displaystyle\frac{\partial z}{\partial \nu^{\beta}}= \sqrt{\gamma} w & \hbox{in}\quad \Omega\times\{y=0\},\\
        w=z= 0 & \hbox{on}\quad \partial_L\mathcal{C}_{\Omega},
        \end{array}
        \right.
        \tag{$\widetilde{S}_{\gamma}^{\beta}$}
\end{equation}
whose associated functional is
\begin{equation*}
\begin{split}
    \Phi_{\g}^{\beta}(w,z)&=\frac{\kappa_{\beta}}{2}\int_{\mathcal{C}_{\Omega}}y^{1-2\beta}|\nabla w|^2dxdy+\frac{\kappa_{\beta}}{2}\int_{\mathcal{C}_{\Omega}}y^{1-2\beta}|\nabla z|^2dxdy-\sqrt{\gamma}\int_{\Omega}w(x,0)z(x,0)dx\\&-\frac{1}{p+1}\int_{\Omega}|w(x,0)|^{p+1}dx.
        \end{split}
\end{equation*}
Since the extension function is an isometry, critical points for the functional $\Phi_{\g}^{\beta}$
in $\mathcal{X}_0^{\beta}(\mathcal{C}_{\Omega})\times\mathcal{X}_0^{\beta}(\mathcal{C}_{\Omega})$ correspond to critical points of $\mc{J}_{\g}^{\beta}$
in $H_0^{\beta}(\Omega)\times H_0^{\beta}(\Omega)$. Moreover, arguing as in \cite[Proposition 3.1]{BCdPS},  the minima of $\Phi_{\g}^{\beta}$
also correspond to the minima of $\mc{J}_{\g}^{\beta}$. Similarly, the extension system of system \eqref{sistemaab} reads as
\begin{equation}\label{extension_systemab}
        \left\{
        \begin{array}{ll}
        -{\rm div}(y^{1-2\mu}\nabla w)= 0 & \hbox{in}\quad \mathcal{C}_{\Omega}, \\
        -{\rm div}(y^{1-2\beta}\nabla z)= 0 & \hbox{in}\quad \mathcal{C}_{\Omega}, \\
        \displaystyle\frac{1}{\gamma^{1-\beta/\alpha}}\frac{\partial w}{\partial \nu^{\mu}}= z+\frac{1}{\gamma^{1-\beta/\alpha}}|w|^{p-1}w & \hbox{in}\quad \Omega\times\{y=0\},\\
        \displaystyle\frac{1}{\gamma^{\beta/\alpha}}\frac{\partial z}{\partial \nu^{\beta}}=w & \hbox{in}\quad \Omega\times\{y=0\},\\
        w=z=0 & \hbox{on}\quad \partial_L\mathcal{C}_{\Omega},
        \end{array}
        \right.
        \tag{$\widetilde{S}_{\gamma}^{\alpha,\beta}$}
\end{equation}
whose associated functional is
\begin{equation*}
\begin{split}
    \Phi_{\g}^{\alpha,\beta}(w,z)=&\frac{\kappa_{\mu}}{2\gamma^{1-\beta/\alpha}}\int_{\mathcal{C}_{\Omega}}y^{1-2\mu}|\nabla w|^2dxdy+\frac{\kappa_{\beta}}{2\gamma^{\beta/\alpha}}\int_{\mathcal{C}_{\Omega}}y^{1-2\beta}|\nabla z|^2dxdy-\int_{\Omega}w(x,0)z(x,0)dx\\&-\frac{1}{(p+1)\gamma^{1-\beta/\alpha}}\int_{\Omega}w(x,0)^{p+1}dx.
        \end{split}
\end{equation*}
Once again, since the extension function is an isometry, critical points of $\Phi_{\g}^{\alpha,\beta}$ in
$\mathcal{X}_0^{\mu}(\mathcal{C}_{\Omega})\times\mathcal{X}_0^{\beta}(\mathcal{C}_{\Omega})$ correspond to critical points of $\mc{J}_{\g}^{\alpha,\beta}$
in $H_0^{\mu}(\Omega)\times H_0^{\beta}(\Omega)$, and also, minima of $\Phi_{\g}^{\alpha,\beta}$ correspond to minima of $\mc{J}_{\g}^{\alpha,\beta}$.

Before finishing this introductory section, let us observe that problem \eqref{ecuacion} can be seen as a linear perturbation of the critical problem,
\begin{equation} \label{crBC}
        \left\{
        \begin{array}{ll}
        (-\D)^{\mu}u=|u|^{2_{\mu}^*-2}u & \hbox{in} \quad\Omega, \\
        u=0 & \hbox{on}\quad \partial\Omega,
        \end{array}
        \right.
\end{equation}
for which, after applying a Pohozaev-type result \cite[Proposition 5.5]{BrCdPS}, one can prove the non-existence of positive solutions under the star-shapeness assumption on the domain $\Omega$.
Moreover, the limit case $\beta\to0$ in problem \eqref{ecuacion} corresponds to
\begin{equation}\label{bezero}
        \left\{
        \begin{tabular}{rll}
        $(-\D)^{\alpha}u=$ &\!\!\!$\g u+|u|^{2_{\alpha}^*-2}u$ &in $\Omega$, \\
        $u=$  &\!\!\!$0$ &on $\partial\Omega$,
        \end{tabular}
        \right.\quad\hbox{with}\quad 0<\alpha<1,
\end{equation}
which was studied in \cite{BCdPS}, where the existence of positive solutions is proved for $N\geq4\alpha$ if and only if $0<\gamma<\lambda_1^*$, with
$\lambda_1^*$ being first eigenvalue of the $(-\Delta)^{\alpha}$ operator under homogeneous Dirichlet boundary conditions.
Note that in our situation the non-local term $\gamma(-\D)^{-\beta}u$ plays actually the role of $\g u$ in \cite{BCdPS}.
\vspace{0.4cm}

\underline{\bf Main results.}
%%%%%%%%%%%%%%%%%%%%%%%%%%%%%%%%%%%%%%%%%%%%%%%%%%%%%%%%%%%%%%%%%%%%
%%%%%%%%%%%%%%%%%%%%%%%%%%%%%%%%%%%%%%%%%%%%%%%%%%%%%%%%%%%%%%%%%%%%
%%%%%%%%%%%%%%%%%%%%%%%%%%%%%%%%%%%%%%%%%%%%%%%%%%%%%%%%%%%%%%%%%%%%
We ascertain the existence of positive solutions for the problem \eqref{ecuacion} depending on the positive real parameter $\g$.
To do so, we will first show the interval of the parameter $\g$ for which there is the possibility of having positive solutions.
Then, we use the equivalence between \eqref{ecuacion} and the systems \eqref{sistemabb} and \eqref{sistemaab} together with the extension technique to prove the main results of this work. Indeed, using the well-known Mountain Pass Theorem (MPT) \cite{AR}, we will prove that there exists a positive solution for \eqref{ecuacion} for any
$$0<\g<\lambda_1^*,$$
where $\lambda_1^*$ is the first eigenvalue of the operator $(-\Delta)^{\alpha}$ under homogeneous Dirichlet boundary conditions. If $1<p+1<2_{\mu}^*$ one might apply the MPT directly since, as we will show,
our problem possesses the mountain pass geometry and thanks to the compact
embedding \eqref{compact_emb} the Palais-Smale condition is satisfied for the functionals
$\mc{F}_\g$, $\mc{J}_{\g}^{\beta}$ and $\mc{J}_{\g}^{\alpha,\beta}$ (see details below in Section \ref{Sec:ProofTh0}). However, at the critical exponent $p=2_{\mu}^*-1$, the compactness of the
Sobolev embedding is lost and the problem becomes very delicate. To overcome this lack of compactness we apply a concentration-compactness argument relying on \cite[Theorem 5.1]{BCdPS}, which is an adaptation to the fractional setting of the classical result of P.-L. Lions, \cite{Lions}. Then we are capable of proving that, under certain conditions, the Palais-Smale condition is satisfied for the functionals $\Phi_{\g}^{\beta}$ and $\Phi_{\g}^{\alpha,\beta}$. Thus, by the arguments above, the result will also follow for the functionals $\mc{F}_\g$, $\mc{J}_{\g}^{\beta}$
and $\mc{J}_{\g}^{\alpha,\beta}$. Consequently, we state now the main results of this paper.

\begin{theorem}
\label{Th0}
Assume $1<p<2_{\mu}^*-1$. Then, for every $\g\in (0,\l_1^*)$, where $\lambda_1^*$ is the first eigenvalue of $(-\Delta)^{\alpha}$ under homogeneous Dirichlet boundary conditions,
there exists a positive solution for the problem \eqref{ecuacion}.
\end{theorem}

\begin{theorem}
\label{Th1}
Assume $p=2_{\mu}^*-1$. Then, for every $\g\in (0,\l_1^*)$, where $\lambda_1^*$ is the first eigenvalue of $(-\Delta)^{\alpha}$ under homogeneous Dirichlet boundary conditions, there exists a positive solution for
the problem \eqref{ecuacion} provided that $N>4\alpha-2\beta$.
\end{theorem}

Let us observe that, even though problem \eqref{ecuacion} is a
non-local but also a linear perturbation of the critical problem
\eqref{crBC}, Theorem \ref{Th1} addresses dimensions
$N>4\alpha-2\beta$, in contrast to the existence result
\cite[Theorem 1.2]{BCdPS} about the linear perturbation
\eqref{bezero}, that covers the range $N\geq4\alpha$. In other
words, the non-local term $(-\Delta)^{-\beta}u$, despite of being
just a linear perturbation, has an important effect on the
dimensions for which the classical Brezis--Nirenberg technique (see
\cite{BN}) based on the minimizers of the Sobolev constant still
works. See details in Section \ref{Subsec:concentracion_compacidad}.
%%%%%%%%%%%%%%%%%%%%%%%%%%%%%%%%%%%%%%%%%%%%%%%%%%%%%%%%%%%%%%%%%%%%
%%%%%%%%%%%%%%%%%%%%%%%%%%%%%%%%%%%%%%%%%%%%%%%%%%%%%%%%%%%%%%%%%%%%
%%%%%%%%%%%%%%%%%%%%%%%%%%%%%%%%%%%%%%%%%%%%%%%%%%%%%%%%%%%%%%%%%%%%
\section{Sub-critical case. Proof Theorem \ref{Th0} }\label{Sec:ProofTh0}
%%%%%%%%%%%%%%%%%%%%%%%%%%%%%%%%%%%%%%%%%%%%%%%%%%%%%%%%%%%%%%%%%%%%
%%%%%%%%%%%%%%%%%%%%%%%%%%%%%%%%%%%%%%%%%%%%%%%%%%%%%%%%%%%%%%%%%%%%
%%%%%%%%%%%%%%%%%%%%%%%%%%%%%%%%%%%%%%%%%%%%%%%%%%%%%%%%%%%%%%%%%%%%
\noindent
In this section we carry out the proof of Theorem \ref{Th0}. This is done through the equivalence between problem \eqref{ecuacion} and systems \eqref{sistemabb} and \eqref{sistemaab}. We note
that the results proved in the sequel for the functionals $\mc{F}_\g$, $\mc{J}_{\g}^{\beta}$ and $\mc{J}_{\g}^{\alpha,\beta}$
translate immediately in analogous results for the functionals $\Phi_{\g}^{\beta}$ and $\Phi_{\g}^{\alpha,\beta}$.
First, we characterize the existence of positive solutions for problem  \eqref{ecuacion} in terms of the parameter $\gamma$. Moreover, for such
characterization
the following eigenvalue problem will be considered
\begin{equation}\label{eiglin1}
        \left\{
        \begin{array}{ll}
        (-\D)^{\mu} u = \l (-\D)^{-\beta} u& \hbox{in} \quad\Omega, \\
        u=0 & \hbox{on}\quad \partial\Omega.
        \end{array}
        \right.
\end{equation}
Then, for the first eigenfunction $\phi_1$ of \eqref{eiglin1}, associated with the first eigenvalue $\l_1^*$, we find
$$\int_\O |(-\D)^{\frac{\mu}{2}} \phi_1|^2dx =\l_1^* \int_\O |(-\D)^{-\frac{\beta}{2}} \phi_1|^2dx,$$
and, therefore,
\begin{equation}\label{bieigen}
\l_1^*=\inf_{u\in H_0^{\mu}(\O)} \frac{\int_\O |(-\D)^{\frac{\mu}{2}} u|^2dx}{ \int_\O |(-\D)^{-\frac{\beta}{2}} u|^2dx}.
\end{equation}
On the other hand, thanks to the definition of the fractional operator $(-\Delta)^{\mu}$, we have that $\phi_1\equiv\varphi_1$,
with $\varphi_1$ as the first eigenfunction of the Laplace operator under homogeneous Dirichlet boundary conditions. Then,
$$(-\D)^{\mu}\phi_1=(-\D)^{\mu}\varphi_1=\lambda_1^{\mu}\varphi_1 \quad\hbox{and}\quad  (-\D)^{-\beta}\phi_1=(-\D)^{-\beta}\varphi_1=\lambda_1^{-\beta}\varphi_1,$$
with $\lambda_1$ as the first eigenvalue of the Laplace operator
under homogeneous Dirichlet boundary conditions. Hence, due to
\eqref{eiglin1}, we conclude that
$\l_1^*=\l_1^{\mu+\beta}=\l_1^{\alpha}$. Thus, $\lambda_1^*$
coincides with the first eigenvalue of the operator
$(-\Delta)^{\alpha}$ under homogeneous Dirichlet or Navier boundary
conditions, depending on whether $\alpha\leq1$ or $1<\alpha<\beta+1$
respectively. As a consequence, we have the following.
\begin{lemma}\label{cota}
Problem \eqref{ecuacion} does not possess a positive solution when
$$\gamma \geq \l_1^*.$$
\end{lemma}
\begin{proof}
Assume that $u$ is a positive solution of \eqref{ecuacion} and let $\varphi_1$ be a positive first eigenfunction of the Laplace operator in $\O$ under homogeneous Dirichlet boundary conditions.
Taking $\varphi_1$ as a test function for equation \eqref{ecuacion} we obtain
\begin{align*}
\lambda_1^{\mu}\int_{\Omega}u\varphi_1dx=\int_{\Omega}\varphi_1(-\Delta)^{\mu}udx &=\gamma\int_{\O}\varphi_1(-\Delta)^{-\beta}udx+\int_{\O}|u|^{p-1}u\varphi_1dx\\
&> \gamma\int_{\O}\varphi_1(-\Delta)^{-\beta}udx=\gamma\int_{\O}u(-\Delta)^{-\beta}\varphi_1dx\\
&=\frac{\gamma}{\lambda_1^{\beta}}\int_{\Omega}u\varphi_1dx.
\end{align*}
Hence, $\lambda_1^{\mu}>\frac{\gamma}{\lambda_1^{\beta}}$, and we conclude that $\gamma<\lambda_1^{\mu+\beta}=\lambda_1^{\alpha}=\lambda_1^*$, proving the lemma.
\end{proof}
Next we check that $\mc{F}_\g$, as well as $\mc{J}_{\g}^{\beta}$ and $\mc{J}_{\g}^{\alpha,\beta}$ satisfy the MP geometry.
\begin{lemma}
\label{lezero}
The functionals $\mc{F}_\g$, $\mc{J}_{\g}^{\beta}$ and $\mc{J}_{\g}^{\alpha,\beta}$ have the MP geometry .
\end{lemma}
\begin{proof}
For short, we prove the result for $\mc{F}_\g$, for the remaining functionals the result follows in a similar way. Without loss of generality,
we consider a function $g\in H_0^{\mu}(\O)$ such that $\|g\|_{p+1}=1$. Because of \eqref{bieigen}, the fractional Sobolev inequality \eqref{sobolev} and
\eqref{eqnorma}, we find that for $t>0$,
\begin{align*}
\mc{F}_\g(tg)&=\frac{t^2}{2}\int_{\Omega}|(-\Delta)^{\frac{\mu}{2}}g|^2dx-\frac{\gamma t^2}{2}\int_{\Omega}|(-\Delta)^{-\frac{\beta}{2}}g|^2dx-\frac{t^{p+1}}{p+1}\\
&\geq\frac{t^2}{2}\int_{\Omega}|(-\Delta)^{\frac{\mu}{2}}g|^2dx-\frac{\gamma t^2}{2\lambda_1^{*}}\int_{\Omega}|(-\Delta)^{\frac{\mu}{2}}g|^2dx-\frac{t^{p+1}}{p+1}\\
&\geq\frac{t^2}{2}\left(1-\frac{\g}{\l_1^*}\right)\int_{\Omega}|(-\Delta)^{\frac{\mu}{2}}g|^2dx-\frac{t^{p+1}}{C(p+1)}\int_{\Omega}|(-\Delta)^{\frac{\mu}{2}}g|^2dx\\
&=\|g\|_{H_0^{\mu}(\O)}^2\left(\frac{1}{2}\left(1-\frac{\gamma}{\lambda_1^*}\right)t^2-\frac{1}{C(p+1)}t^{p+1}\right)>0,
\end{align*}
for $t>0$ sufficiently small and $C>0$ is a constant coming from inequality \eqref{sobolev}, that is,
$$0<t^{p-1}<\frac{C(p+1)}{2}\left(1-\frac{\gamma}{\lambda_1^*}\right).$$
Thus, the functional $\mc{F}_\g$ has a local minimum at $u=0$, i.e., $\mc{F}_\g(tg)>\mc{F}_\g(0)=0$ for any $g\in H_0^{\mu}(\O)$ provided $t>0$ is small enough. Furthermore, it is clear that
\begin{align*}
    \mc{F}_\g(tg)&=\frac{t^2}{2} \int_\O |(-\Delta)^{\frac{\mu}{2}} g|^2dx - \frac{\g t^2}{2} \int_\O |(-\Delta)^{-\frac{\beta}{2}} g|^2dx-\frac{t^{p+1}}{p+1}\\
        &\leq \frac{t^2}{2}\|g\|_{H_0^{\mu}(\O)}^2-\frac{t^{p+1}}{p+1}.
\end{align*}
Then, $\mc{F}_\g(tg) \rightarrow -\infty$ as $t\to \infty$ and, thus, there exists $\hat u \in H_0^{\mu}(\O)$ such that $\mc{F}_\g(\hat u)<0$. Hence, the functional $\mc{F}_\g$ has the mountain pass geometry.\newline
\end{proof}
Similarly we have the MP geometry for the extended functionals.
\begin{lemma}\label{lezeroextension}
The functionals $\Phi_{\g}^{\beta}$ and $\Phi_{\g}^{\alpha,\beta}$ have the MP geometry.
\end{lemma}
\begin{proof}
The proof is similar to the proof of Lemma \ref{lezero}, we only need to note that, thanks to the isometry \eqref{isometry} and the trace inequality
\eqref{sobext}, the extension function minimizes the norm $\|\cdot\|_{\mathcal{X}_0^{\mu}(\mathcal{C}_{\O})}$ among all the functions with the same trace on $\{y=0\}$, i.e.,
$$\|E_{\mu}[\varphi(\cdot,0)]\|_{\mathcal{X}_0^{\mu}(\mathcal{C}_{\O})}\leq\|\varphi\|_{\mathcal{X}_0^{\mu}(\mathcal{C}_{\O})}\quad \hbox{for all}\quad \varphi\in\mathcal{X}_0^{\mu}(\mathcal{C}_{\O}).$$
Therefore,
\begin{equation}\label{min_eig}
\lambda_1^{\mu}=\inf_{\substack{u\in H_0^{\mu}(\O)\\ u\not\equiv0}}
\frac{\|u\|_{H_0^{\mu}(\O)}^2}{\|u\|_{L^{2}(\O)}^2}
=\inf_{\substack{w\in \mathcal{X}_0^{\mu}(\mathcal{C}_{\O})\\
w\not\equiv0}}\frac{\|w\|_{\mathcal{X}_0^{\mu}(\mathcal{C}_{\O})}^2}{\|w(\cdot,0)\|_{L^{2}(\O)}^2}.
\end{equation}
Thus, following  the arguments in the proof of Lemma \ref{lezero}, the result follows.
\end{proof}
%%%%%%%%%%%%%%%%%%%%%%%%%%%%%%%%%%%%%%%%%%%%%%%%%
\begin{definition}\label{def_PS}
Let $V$ be a Banach space. We say that $\{u_n\} \subset V$ is a Palais-Smale (PS) sequence for a functional $\mf{F}$ if
\begin{equation}\label{convergencia}
\mf{F}(u_n)\quad\hbox{is bounded and}\quad  \mf{F}'(u_n) \to 0\quad\mbox{in}\ V'\quad \hbox{as}\quad n\to \infty,
\end{equation}
where $V'$ is the dual space of $V$. Moreover, we say that $\{u_n\}$ satisfies a PS condition if
\begin{equation}\label{conPS}
\{u_n\}\quad \mbox{has a strongly convergent subsequence.}
\end{equation}
\end{definition}
In particular, we say that the functional $\mf{F}$ satisfies the PS condition at level $c$ if every PS sequence at level $c$ for $\mf{F}$
satisfies the PS condition.
In the subcritical range, $1\le p<2_\mu^*-1$, the PS condition is satisfied at any level c due to the compact embedding \eqref{compact_emb}.
However, at the critical exponent $2_{\mu}^*$ the compactness in the Sobolev embedding is lost and, as we will see, the PS condition will be satisfied
only for levels below certain critical level $c^*$.
\begin{lemma}\label{acotacion_ecuacion}
Let $\{u_n\}\subset H_0^\mu(\O)$ be a PS sequence at level $c$ for the functional $\mc{F}_\g$, i.e.
$$\mc{F}_\g(u_n) \rightarrow c,\quad    \mc{F}_\g'(u_n) \rightarrow 0,\quad \hbox{as}\quad n\to \infty.$$
Then, $\{u_n\}$ is bounded in $H_0^{\mu}(\O)$.
\end{lemma}
\begin{proof}
Since $\mc{F}_\g'(u_n) \rightarrow 0$ in $\left(H_0^{\mu}(\O)\right)'$ and
 $\mc{F}_\g(u_n) \to c$, we find that
$$\mc{F}_\g(u_n)-\frac{1}{p+1} \langle \mc{F}_\g'(u_n)|u_n\rangle=c+o(1)\cdot\|u_n\|_{H_0^{\mu}(\O)}.$$
That is,
\begin{align*}
\left(\frac{1}{2}-\frac{1}{p+1}\right)\!\int_\O |(-\Delta)^{\frac{\mu}{2}} u_n|^2dx-\left(\frac{1}{2}-\frac{1}{p+1}\right)\!\int_\O
|(-\Delta)^{-\frac{\beta}{2}}u_n|^2dx  =c+o(1)\cdot\|u_n\|_{H_0^{\mu}(\O)}.
\end{align*}
Therefore, by \eqref{bieigen},
since $\gamma<\lambda_1^*$, using \eqref{eqnorma} we conclude that
$$0<\left(\frac{1}{2}-\frac{1}{p+1}\right)\left(1-\frac{\gamma}{\lambda_1^*}\right)\|u_n\|_{H_0^{\mu}(\O)}^2\leq c+o(1)\cdot\|u_n\|_{H_0^{\mu}(\O)}.$$
Thus, the sequence $\{u_n\}$ is bounded in $H_0^{\mu}(\O)$.
\end{proof}
Following similar ideas as in the above proof, we obtain the following two results.
%%%%%%%%%%%%%%%%%%%%%%%%%%%%%%%%%%%%%%%%%%%%%%%%%%%%%%%%%%%%%%%%%%%%%%%%%%%%%%%%%%%%%%%%%%%%%%%%%%%%%%%%%%%%%%%%%%%%%%%%%%%%%%%%%%%%%%%%%%%%%%%%%%%%%%%%%%%%%%
\begin{lemma}\label{acotacion_sistemabb}
Let $\{(u_n,v_n)\}$ be a PS sequence at level $c$ for the functional $\mc{J}_{\g}^{\beta}$, i.e.
$$\mc{J}_{\g}^{\beta}(u_n,v_n) \rightarrow c,\quad    \left(\mc{J}_{\g}^{\beta}\right)'(u_n,v_n) \rightarrow 0,\quad \hbox{as}\quad n\to \infty.$$
Then, $\{(u_n,v_n)\}$ is bounded in $H_0^{\beta}(\O)\times H_0^{\beta}(\O)$.
\end{lemma}
\begin{lemma}\label{rem}
Let $\{(w_n,z_n)\}$ be a PS sequence at level $c$ for the functional $\Phi_{\g}^{\beta}\ ($resp. for the functional $\Phi_{\g}^{\alpha,\beta})$. Then, $\{(w_n,z_n)\}$ is bounded in $\mathcal{X}_0^{\beta}(\mathcal{C}_{\Omega})\times\mathcal{X}_0^{\beta}(\mathcal{C}_{\Omega})\ ($resp. in $\mathcal{X}_0^{\mu}(\mathcal{C}_{\Omega})\times\mathcal{X}_0^{\beta}(\mathcal{C}_{\Omega}))$
\end{lemma}
Now, we are able to prove one of the main results of this paper.
\begin{proof}[Proof of Theorem \ref{Th0}.]\hfill\break
Since we are dealing with the subcritical case $1<p<2_{\mu}^*-1$, given a PS sequence $\{u_n\}\subset H_0^{\mu}(\O)$ for the functional $\mc{F}_\g$, thanks to
Lemma \ref{acotacion_ecuacion} and the compact inclusion \eqref{compact_emb}, the PS condition is satisfied. Moreover, by Lemma \ref{lezero},
the functional $\mc{F}_\g$ satisfies the MP geometry. Then, due to the MPT \cite{AR} and the PS condition,
the functional $\mc{F}_\g$ possesses a critical point $u\in H_0^{\mu}(\O)$.
Moreover, if we define the set of paths
between the origin and $\hat u$,
$$\G:=\{g\in C([0,1],H_0^{\mu}(\O))\,;\, g(0)=0,\; g(1)=\hat u\},$$
with $\hat u$ given as in Lemma \ref{lezero}, i.e. $\mc{F}_\g(\hat u)<0$, then,
$$\mc{F}_\g(u)=\inf_{g\in\G} \max_{\theta \in [0,1]} \mc{F}_\g(g(\theta))= c.$$
To show that $u>0$, let us consider the functional,
\begin{equation*}
\mc{F}_\g^+(u)=\mc{F}_\g(u^+),
\end{equation*}
where $u^+=\max\{u,0\}$. Repeating with minor changes the arguments carried out above, one readily shows that what was proved for the functional $\mc{F}_\g$
still holds for the functional $\mc{F}_\g^+$. Hence, it follows that $u\geq 0$ and by the Maximum Principle (see \cite{CaSi}), $u>0$.
\end{proof}
\begin{remark}
Once we have proved the existence of a positive solution to problem \eqref{ecuacion}, due to the equivalence between \eqref{ecuacion} and systems
\eqref{sistemabb} and \eqref{sistemaab} we have the existence of a positive solution to both systems too.
\end{remark}
\section{Concentration-Compactness at the critical exponent.}\label{Subsec:concentracion_compacidad}
%%%%%%%%%%%%%%%%%%%%%%%%%%%%%%%%%%%%%%%%%%%%%%%%%
In this subsection we focus on the critical exponent case, $p=2_{\mu}^*-1$, proving Theorem\;\ref{Th1}.
Our aim is to prove the PS condition for the functional $\mc{F}_{\g}$ since the rest of the proof will be similar to what we performed in the previous
section for the subcritical case.

First, by means of a
concentration-compactness argument, we will prove that the PS
condition is satisfied at levels below certain critical level $c^*$
(to be determined). Next, we construct an appropriate path whose energy is
below that critical level $c^*$ and finally we will find a corresponding sequence satisfying the PS condition. Both steps are strongly based on
the use of particular test functions. Hence, through this subsection
we will focus on working with the extended
functionals $\Phi_{\g}^{\beta}$ and $\Phi_{\g}^{\alpha,\beta}$. Once
we have completed this task, since the $\beta$-harmonic extension is
an isometry, given a PS sequence
$\{(w_n,z_n)\}\subset\mathcal{X}_0^{\beta}(\mathcal{C}_{\O})\times\mathcal{X}_0^{\beta}(\mathcal{C}_{\O})$
at level $c$ for the functional $\Phi_{\g}^{\beta}$, satisfying the
PS condition, it is clear that the trace sequence
$\{(u_n,v_n)\}=\{Tr[w_n],Tr[z_n]\}$ belongs to
$H_0^{\beta}(\O)\times H_0^{\beta}(\O)$ and is a PS sequence at
the same level $c$ below certain $c^*$ for the functional
$\mc{J}_{\g}^{\beta}$, satisfying the PS condition. Thus, the
functional $\mc{J}_{\g}^{\beta}$ satisfies the PS condition at every
level $c$ below the critical level $c^*$.  In a similar way we can
infer that the functional $\mc{J}_{\g}^{\alpha,\beta}$ satisfies the corresponding
PS condition.

More specifically, by means of a concentration-compactness argument we first prove that the
PS condition is satisfied for any level $c$ with
\begin{equation}\label{levelbeta}
c<\left(\frac{1}{2}-\frac{1}{2_{\beta}^*}\right)\left(\kappa_{\beta}S(\beta,N)\right)^{\frac{2_{\beta}^*}{2_{\beta}^*-2}}=\frac{\beta}{N} \left(\kappa_{\beta}S(\beta,N)\right)^{\frac{N}{2\beta}},
\tag{$c_{\beta}^*$}
\end{equation}
when dealing with the functional $\Phi_{\g}^{\beta}$, and for any level
\begin{equation}\label{levelmu}
c<\frac{1}{\gamma^{1-\beta/\alpha}}\left(\frac{1}{2}-\frac{1}{2_{\mu}^*}\right)\left(\kappa_{\mu}S(\mu,N)\right)^{\frac{2_{\mu}^*}{2_{\mu}^*-2}}=\frac{1}{\gamma^{1-\beta/\alpha}}\frac{\mu}{N} \left(\kappa_{\mu}S(\mu,N)\right)^{\frac{N}{2\mu}},
\tag{$c_{\mu}^*$}
\end{equation}
when dealing with the functional $\Phi_{\g}^{\alpha,\beta}$. Next, using an appropriate cut-off version of the extremal functions \eqref{u_eps}
we will obtain a path below the critical levels $c_{\beta}^*$ and $c_{\mu}^*$.\newline
%%%%%%%%%%%%%%%%%%%%%%%%%%%%%%%%%%%%%%%%%%%%%%%%%%%%
\subsection{PS condition under a critical level}
%%%%%%%%%%%%%%%%%%%%%%%%%%%%%%%%%%%%%%%%%%%%%%%%%%%%
To accomplish the first step, let us start recalling the following.
\begin{definition}
We say that a sequence $\{y^{1-2\mu}|\nabla w_n|^2\}_{n\in\mathbb{N}}$ is tight if for any $\eta>0$ there exists $\rho_0>0$ such that
\begin{equation*}
\int_{\{y>\rho_0\}}\int_{\Omega}y^{1-2\mu}|\nabla w_n|^2dxdy\leq\eta,\quad\forall n\in\mathbb{N}.
\end{equation*}
\end{definition}
\noindent In particular, since we are dealing with a system, we say that the sequence
$$\{(y^{1-2\mu}|\nabla w_n|^2,y^{1-2\beta}|\nabla z_n|^2)\}_{n\in\mathbb{N}},$$ is tight if for any $\eta>0$ there exists $\rho_0>0$ such that
\begin{equation*}
\int_{\{y>\rho_0\}}\int_{\Omega}y^{1-2\mu}|\nabla w_n|^2dxdy+\int_{\{y>\rho_0\}}\int_{\Omega}y^{1-2\beta}|\nabla z_n|^2dxdy\leq\eta,\quad\forall n\in\mathbb{N}.
\end{equation*}
Now we state the Concentration-Compactness Theorem  \cite[Theorem 5.1]{BCdPS} that will be useful in the proof of the PS condition.
\begin{theorem}\label{th:concentracion}
Let $\{w_n\}$ be a weakly convergent sequence to $w$ in $\mathcal{X}_0^{\mu}(\mathcal{C}_{\Omega})$ such that the sequence $\{y^{1-2\mu}|\nabla w_n|^2\}_{n\in\mathbb{N}}$ is tight. Let $u_n=w_n(x,0)$ and $u=w(x,0)$. Let $\nu,\ \zeta$ be two nonnegative measures such that
\begin{equation*}
y^{1-2\mu}|\nabla w_n|^2\to\zeta\quad\mbox{and}\quad|u_n|^{2_{\mu}^*}\to\nu,\quad\mbox{as}\ n\to\infty
\end{equation*}
in the sense of measures. Then there exists an index set $I$, at most countable, points $\{x_i\}_{i\in I}\subset\Omega$ and positive numbers $\nu_i$, $\zeta_i$, with $i\in I$, such that,
\begin{itemize}
\item $\nu=|u|^{2_{\mu}^*}+\sum\limits_{i\in I}\nu_i\delta_{x_i},\ \nu_i>0,$
\item $\zeta=y^{1-2\mu}|\nabla w|^2+\sum\limits_{i\in I}\zeta_i\delta_{x_i},\ \zeta_i>0,$
\end{itemize}
where $\delta_{x_{j}}$ stands for the Dirac's delta centered at $x_j$ and satisfying the condition
\begin{equation*}
\zeta_i\geq S(\mu,N)\nu_i^{2/2_{\mu}^*}.
\end{equation*}
\end{theorem}
With respect to the PS condition we have the following.
\begin{lemma}\label{PScondition_extensionsistemabb}
If $p=2_{\beta}^*-1$ the functional $\Phi_{\g}^{\beta}$ satisfies the PS condition for any level $c$ below the critical level defined by \eqref{levelbeta}.
\end{lemma}

\begin{proof}\renewcommand{\qedsymbol}{}
Let $\{(w_n,z_n)\}_{n\in\mathbb{N}}\subset \mathcal{X}_0^{\beta}(\mathcal{C}_{\Omega})\times \mathcal{X}_0^{\beta}(\mathcal{C}_{\Omega})$ be a PS sequence at level $c$ for the functional  $\Phi_{\g}^{\beta}$, i.e.
\begin{equation}\label{critic}
\Phi_{\g}^{\beta}(w_n,z_n)\to c<c_{\beta}^*\quad\mbox{and}\quad \left(\Phi_{\g}^{\beta}\right)'(w_n,z_n)\to 0.
\end{equation}
From \eqref{critic} and Lemma \ref{rem} we get that the sequence $\{(w_n,z_n)\}_{n\in\mathbb{N}}$ is uniformly bounded in $\mathcal{X}_0^{\beta}(\mathcal{C}_{\Omega})\times \mathcal{X}_0^{\beta}(\mathcal{C}_{\Omega})$, in other words, there exists a finite $M>0$ such that
\begin{equation}
\label{Mfrac}
||w_n||_{\mathcal{X}_0^{\beta}(\mathcal{C}_{\Omega})}^2+||z_n||_{\mathcal{X}_0^{\beta}(\mathcal{C}_{\Omega})}^2\leq M,
\end{equation}
and, as a consequence, we can assume that, up to a subsequence,
\begin{align}\label{conver}
& w_n\rightharpoonup w\quad\mbox{weakly in}\ \mathcal{X}_0^{\beta}(\mathcal{C}_{\Omega}),\nonumber\\
& w_n(x,0) \rightarrow w(x,0)\quad\mbox{strong in}\ L^r(\Omega), \hbox{with}\quad 1\leq r<2_{\beta}^*,\nonumber\\
& w_n(x,0) \rightarrow w(x,0)\quad\mbox{a.e. in}\ \Omega,
\end{align}
and
\begin{align}\label{conver2}
& z_n\rightharpoonup z\quad\mbox{weakly in}\ \mathcal{X}_0^{\beta}(\mathcal{C}_{\Omega}),\nonumber\\
& z_n(x,0) \rightarrow z(x,0)\quad\mbox{strong in}\ L^r(\Omega), 1\leq r<2_{\beta}^*,\nonumber\\
& z_n(x,0) \rightarrow z(x,0)\quad\mbox{a.e. in}\ \Omega.
\end{align}
Before applying Theorem \ref{th:concentracion}, first we need to check that the PS sequence $\{(w_n,z_n)\}_{n\in\mathbb{N}}$ is tight. To avoid any unnecessary technical details,
and since the functional $\Phi_{\g}^{\beta}$ is obtained as a particular case (up to a multiplication by $\sqrt{\g}$) of the functional
$\Phi_{\g}^{\alpha,\beta}$ when $\alpha=2\beta$, we prove the following.
\end{proof}
\begin{lemma}
A PS sequence $\{(w_n,z_n)\}_{n\in\mathbb{N}}\subset \mathcal{X}_0^{\mu}(\mathcal{C}_{\Omega})\times \mathcal{X}_0^{\beta}(\mathcal{C}_{\Omega})$ at level $c$
for the functional $\Phi_{\g}^{\alpha,\beta}$ is tight.
\end{lemma}
\begin{proof}
The proof is similar to the proof of Lemma 3.6 in \cite{BCdPS}, which follows some arguments contained in \cite{AAP}, and we include it for the reader's convenience. By contradiction, suppose that there exists $\eta_0>0$ and $m_0\in\mathbb{N}$ such that for any $\rho>0$ we have, up to a subsequence,
\begin{equation}\label{contradiction}
\int_{\{y>\rho\}}\int_{\Omega}y^{1-2\mu}|\nabla w_n|^2dxdy+\int_{\{y>\rho\}}\int_{\Omega}y^{1-2\beta}|\nabla z_n|^2dxdy>\eta_0,\quad\forall m\geq m_0.
\end{equation}
Let $\varepsilon>0$ be fixed (to be determined later), and let $\rho_0>0$ such that

\begin{equation*}
\int_{\{y>\rho_0\}}\int_{\Omega}y^{1-2\mu}|\nabla w|^2dxdy+\int_{\{y>\rho_0\}}\int_{\Omega}y^{1-2\beta}|\nabla z|^2dxdy<\varepsilon.
\end{equation*}
Let $j=\left[\frac{M}{\varepsilon\kappa}\right]$ be the integer part, with
$\kappa=\min\left\{\frac{\kappa_{\mu}}{\gamma^{1-\beta/\alpha}},\frac{\kappa_{\beta}}{\gamma^{\beta/\alpha}}\right\}$ and
$I_k=\{y\in\mathbb{R}^+:\rho_0+k\leq y\leq \rho_0+k+1\}$, $k=0,1,\ldots,j$. Then, using \eqref{Mfrac},
\begin{align*}
\sum_{k=0}^{j}\int_{I_k}\int_{\Omega}y^{1-2\mu}|\nabla w_n|^2dxdy&+\int_{I_k}\int_{\Omega}y^{1-2\beta}|\nabla z_n|^2dxdy\\
&\leq\int_{\mathcal{C}_{\Omega}}y^{1-2\mu}|\nabla w_n|^2dxdy+\int_{\mathcal{C}_{\Omega}}y^{1-2\beta}|\nabla z_n|^2dxdy\\
&\leq\frac{M}{\kappa}<\varepsilon(j+1).
\end{align*}
Hence, there exists $k_0\in\{0,1,\ldots,j\}$ such that
\begin{equation}\label{menor}
\int_{I_{k_0}}\int_{\Omega}y^{1-2\mu}|\nabla w_n|^2dxdy+\int_{I_{k_0}}\int_{\Omega}y^{1-2\beta}|\nabla z_n|^2dxdy\leq\varepsilon.
\end{equation}
Take now a smooth cut-off function
\begin{equation*}
X(y)=\left\{
\begin{tabular}{lcl}
$0$&&if $y\leq r+k_0$,\\
$1$&&if $y\geq r+k_0+1$,
\end{tabular}
\right.
\end{equation*}
and define $(t_n,s_n)=(X(y)w_n,X(y)z_n)$. Then
\begin{align*}
&\left|\left\langle \left(\Phi_{\g}^{\alpha,\beta}\right)'(w_n,z_n)-\left(\Phi_{\g}^{\alpha,\beta}\right)'(t_n,s_n)\Big|(t_n,s_n)\right\rangle\right|\\
&=\frac{\kappa_{\mu}}{\gamma^{1-\beta/\alpha}}\int_{\mathcal{C}_{\Omega}}y^{1-2\mu}\langle\nabla(w_n-t_n),\nabla t_n\rangle dxdy+\frac{\kappa_{\beta}}{\gamma^{\beta/\alpha}}\int_{\mathcal{C}_{\Omega}}y^{1-2\beta}\langle\nabla(z_n-s_n),\nabla s_n\rangle dxdy\\
&=\frac{\kappa_{\mu}}{\gamma^{1-\beta/\alpha}}\int_{I_{k_0}}\int_{\Omega}y^{1-2\mu}\langle\nabla(w_n-t_n),\nabla t_n\rangle dxdy+\frac{\kappa_{\beta}}{\gamma^{\beta/\alpha}}\int_{I_{k_0}}\int_{\Omega}y^{1-2\beta}\langle\nabla(z_n-s_n),\nabla s_n\rangle dxdy.
\end{align*}
Now, because of the Cauchy-Schwarz inequality, inequality \eqref{menor} and the compact inclusion\footnote{Let us recall that $\beta\in(0,1)$ and $\mu:=\alpha-\beta\in(0,1)$ thus, the weights $w_1(x,y)=y^{1-2\mu}$ and $w_2(x,y)=y^{1-2\beta}$ belongs to the Muckenhoupt class $A_2$. We refer to \cite{FKS} for the precise definition as well as some useful properties of the weights belonging to the Muckenhoupt classes $A_p$.},
\small{\begin{equation*}
H^1(I_{k_0}\times\Omega,y^{1-2\mu}dxdy)\times H^1(I_{k_0}\times\Omega,y^{1-2\beta}dxdy)\hookrightarrow L^2(I_{k_0}\times\Omega,y^{1-2\mu}dxdy)\times L^2(I_{k_0}\times\Omega,y^{1-2\beta}dxdy),
\end{equation*}}
it follows that,
\begin{align*}
&\left|\left\langle \left(\Phi_{\g}^{\alpha,\beta}\right)'(w_n,z_n)-\left(\Phi_{\g}^{\alpha,\beta}\right)'(t_n,s_n)\Big|(t_n,s_n) \right\rangle\right|\\
&\leq\frac{\kappa_{\mu}}{\gamma^{1-\beta/\alpha}}\left(\int_{I_{k_0}}\int_{\Omega}y^{1-2\mu}|\nabla(w_n-t_n)|^2dxdy\right)^{1/2}\left(\int_{I_{k_0}}\int_{\Omega}y^{1-2\mu}|\nabla t_n|^2dxdy\right)^{1/2}\\
&+\frac{\kappa_{\beta}}{\gamma^{\beta/\alpha}}\left(\int_{I_{k_0}}\int_{\Omega}y^{1-2\beta}|\nabla(z_n-s_n)|^2dxdy\right)^{1/2}\left(\int_{I_{k_0}}\int_{\Omega}y^{1-2\beta}|\nabla s_n|^2dxdy\right)^{1/2}\\
&\leq\max\left\{\frac{\kappa_{\mu}}{\gamma^{1-\beta/\alpha}},\frac{\kappa_{\beta}}{\gamma^{\beta/\alpha}}\right\}c\varepsilon \leq C\varepsilon,
\end{align*}
where $C:=c\max\left\{\frac{\kappa_{\mu}}{\gamma^{1-\beta/\alpha}},\frac{\kappa_{\beta}}{\gamma^{\beta/\alpha}}\right\}>0$. On the other hand, by \eqref{critic},
\begin{equation*}
\left|\left\langle
\left(\Phi_{\g}^{\alpha,\beta}\right)'(t_n,s_n)\Big|(t_n,s_n)
\right\rangle\right|\leq c_1\varepsilon+o(1),
\end{equation*}
with $c_1$ a positive constant. Thus, we conclude
\begin{align*}
\int_{\{y>r+k_0+1\}}\int_{\Omega}y^{1-2\mu}|\nabla w_n|^2dxdy&+\int_{\{y>r+k_0+1\}}\int_{\Omega}y^{1-2\beta}|\nabla z_n|^2dxdy\\
&\leq\int_{\mathcal{C}_{\Omega}}y^{1-2\mu}|\nabla t_n|^2dxdy+\int_{\mathcal{C}_{\Omega}}y^{1-2\beta}|\nabla s_n|^2dxdy\\
&\leq\frac{1}{\kappa}\left\langle
\left(\Phi_{\g}^{\alpha,\beta}\right)'(t_n,s_n)\Big|(t_n,s_n)
\right\rangle \leq C \varepsilon,
\end{align*}
in contradiction with \eqref{contradiction}. Hence, the sequence is tight.
\end{proof}
\begin{proof}[Continuation proof Lemma \ref{PScondition_extensionsistemabb}]
Once we have proved that the PS sequence
$$\{(w_n,z_n)\}_{n\in\mathbb{N}}\subset \mathcal{X}_0^{\beta}(\mathcal{C}_{\Omega})\times \mathcal{X}_0^{\beta}(\mathcal{C}_{\Omega}),$$
is tight, we can apply Theorem \ref{th:concentracion}. Consequently, up to a subsequence, there exists an at most countable set $I$,
a sequence of points $\{x_i\}_{i\in I}\subset\Omega$ and non-negative real numbers $\nu_i$ and $\zeta_i$ such that

\begin{itemize}
\item $|u_n|^{2_{\beta}^*}\to \nu=|u|^{2_{\beta}^*}+\sum\limits_{i\in I}\nu_i\delta_{x_i},$
\item $y^{1-2\beta}|\nabla w_n|^2\to\zeta=y^{1-2\beta}|\nabla w|^2+\sum\limits_{i\in I}\zeta_i\delta_{x_i},$
\item $y^{1-2\beta}|\nabla z_n|^2\to\widetilde{\zeta}=y^{1-2\beta}|\nabla z|^2+\sum\limits_{i\in I}\widetilde{\zeta}_i\delta_{x_i},$
\end{itemize}
where $\delta_{x_i}$ is the Dirac's delta centered at $x_i$ and satisfying,
\begin{equation}\label{in:concentracion}
\zeta_i\geq S(\mu,N)\nu_i^{2/2_{\mu}^*}.
\end{equation}
We fix $j\in I$ and we let $\phi\in\mathcal{C}_0^{\infty}(\mathbb{R}_+^{N+1})$ be a non-increasing smooth cut-off function verifying $\phi=1$ in $B_1^+(x_{j})$, $\phi=0$ in $B_2^+(x_{j})^c$, with $B_r^+(x_j)\subset\mathbb{R}^{N}\times\{y\geq0\}$ the $(N+1)$-dimensional semi-ball of radius $r>0$ centered at $x_j$. Let now $\phi_{\varepsilon}(x,y)=\phi(x/\varepsilon,y/\varepsilon)$, such that $|\nabla\phi_{\varepsilon}|\leq\frac{C}{\varepsilon}$ and denote $\Gamma_{2\varepsilon}=B_{2\varepsilon}^+(x_{j})\cap\{y=0\}$. Therefore, since by \eqref{critic}
\begin{equation}\label{tocero}
\left(\Phi_{\g}^{\beta}\right)'(w_n,z_n)\to 0\quad \hbox{in the dual space}\quad  \left(\mathcal{X}_0^{\beta}(\mathcal{C}_{\Omega})\times \mathcal{X}_0^{\beta}(\mathcal{C}_{\Omega})\right)',
\end{equation}
taking the dual product in \eqref{tocero} with $(\phi_{\varepsilon}w_n,\phi_{\varepsilon}z_n)$, we obtain
\begin{equation*}
\begin{split}
\lim_{n\to\infty}&\left(\kappa_{\beta}\int_{\mathcal{C}_{\Omega}}y^{1-2\beta}\nabla w_n\nabla(\phi_{\varepsilon}w_n)dxdy+\kappa_{\beta}\int_{\mathcal{C}_{\Omega}}y^{1-2\beta}\nabla z_n\nabla(\phi_{\varepsilon}z_n)dxdy\right.\\
&\ \ \left.-2\sqrt{\gamma}\int_{\Gamma_{2\varepsilon}}\phi_{\varepsilon}w_n(x,0)z_n(x,0)dx-\int_{\Gamma_{2\varepsilon}}\phi_{\varepsilon}|w_n|^{2_{\beta}^*}(x,0)dx\right)=0.
\end{split}
\end{equation*}
Hence,
\begin{equation*}
\begin{split}
&\lim_{n\to\infty}\left(\kappa_{\beta}\int_{\mathcal{C}_{\Omega}}y^{1-2\beta}\langle\nabla w_n,\nabla\phi_{\varepsilon} \rangle w_ndxdy+\kappa_{\beta}\int_{\mathcal{C}_{\Omega}}y^{1-2\beta}\langle\nabla z_n,\nabla\phi_{\varepsilon} \rangle z_ndxdy\right)\\
&=\lim_{n\to\infty}\left(2\sqrt{\gamma}\int_{\Gamma_{2\varepsilon}}\phi_{\varepsilon}w_n(x,0)z_n(x,0)dx+\int_{\Gamma_{2\varepsilon}}\phi_{\varepsilon}|w_n|^{2_{\beta}^*}(x,0)dx\right.\\
&\ \ \ \ \ \ \ \ \ \ \ \ \left.-\kappa_{\beta}\int_{\mathcal{C}_{\Omega}}y^{1-2\beta}\phi_{\varepsilon}|\nabla w_n|^2dxdy-\kappa_{\beta}\int_{\mathcal{C}_{\Omega}}y^{1-2\beta}\phi_{\varepsilon}|\nabla z_n|^2dxdy\right).
\end{split}
\end{equation*}
Moreover, thanks to \eqref{conver}, \eqref{conver2} and Theorem \ref{th:concentracion}, we find,
\begin{equation}\label{eq:tozero}
\begin{split}
&\lim_{n\to\infty}\left(\kappa_{\beta}\int_{\mathcal{C}_{\Omega}}y^{1-2\beta}\langle\nabla w_n,\nabla\phi_{\varepsilon} \rangle w_ndxdy+\kappa_{\beta}\int_{\mathcal{C}_{\Omega}}y^{1-2\beta}\langle\nabla z_n,\nabla\phi_{\varepsilon} \rangle z_ndxdy\right)\\
&\ \ \ \ \ \ \ =2\sqrt{\gamma}\int_{\Gamma_{2\varepsilon}}\phi_{\varepsilon}w(x,0)z(x,0)dx+\int_{\Gamma_{2\varepsilon}}\phi_{\varepsilon}d\nu-\kappa_{\beta}\int_{B_{2\varepsilon}^+(x_{j})}\phi_{\varepsilon}d\zeta-\kappa_{\beta}\int_{B_{2\varepsilon}^+(x_{j})}\phi_{\varepsilon}d\widetilde{\zeta}.
\end{split}
\end{equation}
Assume for the moment that the left hand side of \eqref{eq:tozero} vanishes as $\varepsilon\to0$. Then, it follows that,
\begin{equation*}
\begin{split}
0&=\lim_{\varepsilon\to0}2\sqrt{\gamma}\int_{\Gamma_{2\varepsilon}}\phi_{\varepsilon}w(x,0)z(x,0)dx+\int_{\Gamma_{2\varepsilon}}\phi_{\varepsilon}d\nu-\kappa_{\beta}\int_{B_{2\varepsilon}^+(x_{j})}\phi_{\varepsilon}d\zeta-\kappa_{\beta}\int_{B_{2\varepsilon}^+(x_{j})}\phi_{\varepsilon}d\widetilde{\zeta}\\
&=\nu_{j}-\kappa_{\beta}\zeta_{j}-\kappa_{\beta}\widetilde{\zeta}_{j},
\end{split}
\end{equation*}
and we conclude,
\begin{equation}\label{eq:compacidad}
\nu_{j}=\kappa_{\beta}\left(\zeta_{j}+\widetilde{\zeta}_{j}\right).
\end{equation}
Finally, we have two options, either the compactness of the PS sequence or concentration around those points $x_j$. In other words, either $\nu_{j}=0$, so that $\zeta_{j}=\widetilde{\zeta}_{j}=0$ or, thanks to \eqref{eq:compacidad} and \eqref{in:concentracion}, $\nu_{j}\geq\left(\kappa_{\beta}S(\beta,N)\right)^{\frac{2_{\beta}^*}{2_{\beta}^*-2}}$. In case of having concentration, we find,
\begin{equation*}
\begin{split}
c=&\lim_{n\to\infty}\Phi_{\g}^{\beta}(w_n,z_n)=\lim_{n\to\infty}\Phi_{\g}^{\beta}(w_n,z_n)-\frac{1}{2}\left\langle\left(\Phi_{\g}^{\beta}\right)'(w_n,z_n)\Big|(w_n,z_n)\right\rangle\\
=& \left(\frac{1}{2}-\frac{1}{2_{\beta}^*}\right)\int_{\O}|w(x,0)|^{2_{\beta}^*}dx+\left(\frac{1}{2}-\frac{1}{2_{\beta}^*}\right)\nu_{k_0}\\
\geq&\left(\frac{1}{2}-\frac{1}{2_{\beta}^*}\right)\left(\kappa_{\beta}S(\beta,N)\right)^{\frac{2_{\beta}^*}{2_{\beta}^*-2}}=c_{\beta}^*,
\end{split}
\end{equation*}
in contradiction with the hypotheses $c<c_{\beta}^*$. It only remains to prove that the left hand side of \eqref{eq:tozero} vanishes as $\varepsilon\to0$. Due to \eqref{critic} and Lemma \ref{rem},
the PS sequence $\{(w_n,z_n)\}_{n\in\mathbb{N}}$ is bounded in $\mathcal{X}_0^{\beta}(\mathcal{C}_{\Omega})\times\mathcal{X}_0^{\beta}(\mathcal{C}_{\Omega})$, so that, up to a subsequence,
\begin{equation*}
\begin{split}
(w_n,z_n)&\rightharpoonup(w,z)\in \mathcal{X}_0^{\beta}(\mathcal{C}_{\Omega})\times\mathcal{X}_0^{\beta}(\mathcal{C}_{\Omega}),\\
(w_n,z_n)&\rightarrow(w,z)\quad\mbox{a.e. in}\quad \mathcal{C}_{\Omega}.
\end{split}
\end{equation*}
Moreover, for $r<2^*=\frac{2(N+1)}{N-1}$ we have the compact inclusion,
\begin{equation*}
\mathcal{X}_0^{\beta}(\mathcal{C}_{\Omega})\times\mathcal{X}_0^{\beta}(\mathcal{C}_{\Omega})\hookrightarrow L^{r}(\mathcal{C}_{\Omega},y^{1-2\beta}dxdy)\times L^{r}(\mathcal{C}_{\Omega},y^{1-2\beta}dxdy).
\end{equation*}
Applying H\"older's inequality with $p=\frac{N+1}{N-1}$ and $q=\frac{N+1}{2}$, we find,
\begin{equation*}
\begin{split}
\int_{B_{2\varepsilon}^+(x_{k_0})}&  y^{1-2\beta}|\nabla\phi_{\varepsilon}|^2|w_n|^2dxdy\\
\leq& \left(\int_{B_{2\varepsilon}^+(x_{k_0})}\!\!\!\!\!\!\!\!\! y^{1-2\beta}|\nabla\phi_{\varepsilon}|^{N+1}dxdy\right)^{\frac{2}{N+1}}\left(\int_{B_{2\varepsilon}^+(x_{k_0})}\!\!\!\!\!\!\!\!\! y^{1-2\beta}|w_n|^{2\frac{N+1}{N-1}}dxdy\right)^{\frac{N-1}{(N+1)}}\\
\leq&\frac{1}{\varepsilon^2}\left(\int_{B_{2\varepsilon}(x_{k_0})}\int_0^\varepsilon y^{1-2\beta}dxdy\right)^{\frac{2}{N+1}}\left(\int_{B_{2\varepsilon}^+(x_{k_0})}\!\!\!\!\!\!\!\!\! y^{1-2\beta}|w_n|^{2\frac{N+1}{N-1}}dxdy\right)^{\frac{N-1}{(N+1)}}\\
\leq& c_0\varepsilon^{\frac{2(1-2\beta)}{N+1}}\left(\int_{B_{2\varepsilon}^+(x_{k_0})}\!\!\!\!\!\!\!\!\! y^{1-2\beta}|w_n|^{2\frac{N+1}{N-1}}dxdy\right)^{\frac{N-1}{(N+1)}}\\
\leq& c_0 \varepsilon^{\frac{2(1-2\beta)}{N+1}}\varepsilon^{\frac{(2+N-2\beta)(N-1)}{(N+1)}}\left(\int_{B_{2}^+(x_{k_0})}y^{1-2\beta}|w_n(\varepsilon x,\varepsilon y)|^{2\frac{N+1}{N-1}}dxdy\right)^{\frac{N-1}{(N+1)}}\\
\leq& c_1 \varepsilon^{N-2\beta}.
\end{split}
\end{equation*}
for appropriate positive constants $c_0$ and $c_1$. In a similar way,
\begin{equation*}
\int_{B_{2\varepsilon}^+(x_{k_0})}\!\!\!\!\!\!\!\!\! y^{1-2\beta}|\nabla\phi_{\varepsilon}|^2|z_n|^2dxdy\leq c_2 \varepsilon^{N-2\beta}.
\end{equation*}
Thus, we find that,
\begin{equation*}
\begin{split}
0\leq&\lim_{n\to\infty}\left|\kappa_{\beta}\int_{\mathcal{C}_{\Omega}}y^{1-2\beta}\langle\nabla w_n,\nabla\phi_{\varepsilon} \rangle w_ndxdy+\kappa_{\beta}\int_{\mathcal{C}_{\Omega}}y^{1-2\beta}\langle\nabla z_n,\nabla\phi_{\varepsilon} \rangle z_ndxdy\right|\\
\leq&\kappa_{\beta}\lim_{n\to\infty}\left(\int_{\mathcal{C}_{\Omega}}y^{1-2\beta}|\nabla w_n|^2dxdy\right)^{1/2}\left(\int_{B_{2\varepsilon}^+(x_{k_0})}y^{1-2s}|\nabla\phi_{\varepsilon}|^2|w_n|^2dxdy\right)^{1/2}\\
+&\kappa_{\beta}\lim_{n\to\infty}\left(\int_{\mathcal{C}_{\Omega}}y^{1-2\beta}|\nabla z_n|^2dxdy\right)^{1/2}\left(\int_{B_{2\varepsilon}^+(x_{k_0})}y^{1-2s}|\nabla\phi_{\varepsilon}|^2|z_n|^2dxdy\right)^{1/2}\\
\leq&C \varepsilon^{\frac{N-2\beta}{2}}\to0,
\end{split}
\end{equation*}
as $\varepsilon\to 0$ and the proof of the Lemma \ref{PScondition_extensionsistemabb} is complete.
\end{proof}
Next we show the corresponding result for the functional $\Phi_{\g}^{\alpha,\beta}$.
\begin{lemma}\label{PScondition_extensionsistemaab}
If $p=2_{\mu}^*-1$ the functional $\Phi_{\g}^{\alpha,\beta}$
satisfies the PS condition for any level $c$ below the critical
level defined by \eqref{levelmu}.
\end{lemma}
The proof of this result is similar to the one of Lemma \ref{PScondition_extensionsistemabb}, so we omit the details for short.

\subsection{PS sequences under a critical level}
%%%%%%%%%%%%%%%%%%%%%%%%%%%%%%%%%%%%%%%%%%%%%%%%%%%%
At this point, it remains to show that we can obtain
PS sequences for the functionals $\Phi_{\g}^{\beta}$ and $\Phi_{\g}^{\alpha,\beta}$ under the critical levels defined by \eqref{levelbeta} and \eqref{levelmu} respectively. To do so, we consider the extremal functions of the fractional Sobolev inequality \eqref{sobolev}, namely, given $\theta\in(0,1)$, we set
\begin{equation*}
u_{\varepsilon}^{\theta}(x)=\frac{\varepsilon^{\frac{N-2\theta}{2}}}{(\varepsilon^2+|x|^2)^{\frac{N-2\theta}{2}}},
\end{equation*}
and $w_{\varepsilon}^{\theta}=E_{\theta}[u_{\varepsilon}^{\theta}]$ its $\theta$-harmonic extension.
Then, since $w_{\varepsilon}^{\theta}$ is a minimizer of the Sobolev inequality, it holds
\begin{equation*}
S(\theta,N)=\frac{\displaystyle\int_{\mathbb{R}_+^{N+1}}y^{1-2\theta}|\nabla w_{\varepsilon}^{\theta} |^2dxdy}{\displaystyle\|u_{\varepsilon}^{\theta}\|_{L^{2_{\theta}^*}(\mathbb{R}^N)}^2}.
\end{equation*}
We take a non-increasing smooth cut-off function $\phi_0(t)\in\mathcal{C}_0^{\infty}(\mathbb{R}_+)$ such that
$$\phi_0(t)=1\quad \hbox{if}\quad 0\leq t\leq1/2\quad \hbox{and}\quad \phi_0(t)=0\quad \hbox{if}\quad t\geq 1.$$
Assume without loss of generality that $0\in\Omega$, $r>0$ small enough such that $\overline{B}_r^+\subseteq\overline{\mathcal{C}}_{\Omega}$, and define the function $\phi_r(x,y)=\phi_0(\frac{r_{x,y}}{r})$ where $r_{xy}=|(x,y)|=\left(|x|^2+y^2\right)^{1/2}$. Note that $\phi_r w_{\varepsilon}^{\theta}\in\mathcal{X}_0^{\theta}(\mathcal{C}_{\Omega})$. We recall now the following lemma proved in \cite{BCdPS}.
\begin{lemma}\label{estcol}
The family $\{\phi_r w_{\varepsilon}^{\theta}\}$ and its trace on $\{y=0\}$, denoted by $\{\phi_r u_{\varepsilon}^{\theta}\}$, satisfy
\begin{equation*}
\begin{split}
\|\phi_r w_{\varepsilon}^{\theta}\|_{\mathcal{X}_0^{\theta}(\mathcal{C}_{\Omega})}^{2}&=\|w_{\varepsilon}^{\theta}\|_{\mathcal{X}_0^{\theta}(\mathcal{C}_{\Omega})}^{2}+O(\varepsilon^{N-2\theta}),\\
\|\phi_r u_{\varepsilon}^{\theta}\|_{L^2(\Omega)}^{2}&=
\left\{
        \begin{tabular}{lc}
        $C \varepsilon^{2\theta}+O(\varepsilon^{N-2\theta})$ & if $N>4\theta$, \\
                $C \varepsilon^{2\theta}|\log(\varepsilon)|$ & if $N=4\theta$.
        \end{tabular}
        \right.
\end{split}
\end{equation*}
\end{lemma}
\begin{remark}
Since $\|u_{\varepsilon}^{\theta}\|_{L^{2_{\theta}^*}(\mathbb{R}^N)}\sim C$ does not depend on $\varepsilon$ it follows that
\begin{equation*}
\|\phi_r u_{\varepsilon}^{\theta}\|_{L^{2_{\theta}^*}(\Omega)}=\|u_{\varepsilon}^{\theta}\|_{L^{2_{\theta}^*}(\mathbb{R}^N)}+O(\varepsilon^N)=C+O(\varepsilon^N).
\end{equation*}
\end{remark}
\noindent Next, we define the normalized functions,
\begin{equation*}
\eta_{\varepsilon}^{\theta}=\frac{\phi_r w_{\varepsilon}^{\theta}}{\|\phi_r u_{\varepsilon}^{\theta}\|_{2_{\theta}^*}}\quad \mbox{and} \quad \sigma_{\varepsilon}^{\theta}=\frac{\phi_r u_{\varepsilon}^{\theta}}{\|\phi_r u_{\varepsilon}^{\theta}\|_{2_{\theta}^*}},
\end{equation*}
then, because of Lemma \ref{estcol} the following estimates hold,
\begin{equation}\label{estimaciones}
\begin{split}
\|\eta_{\varepsilon}^{\theta}\|_{\mathcal{X}_0^{\theta}(\mathcal{C}_{\Omega})}^{2}&=S(\theta,N)+O(\varepsilon^{N-2\theta}),\\
\|\sigma_{\varepsilon}^{\theta}\|_{L^2(\Omega)}^{2}&=
\left\{
        \begin{tabular}{lc}
        $C \varepsilon^{2\theta}+O(\varepsilon^{N-2\theta})$ & if $N>4\theta$, \\
                $C \varepsilon^{2\theta}|\log(\varepsilon)|$ & if $N=4\theta$,
        \end{tabular}
        \right.\\
\|\sigma_{\varepsilon}^{\theta}\|_{L^{2_{\theta}^*}(\Omega)}&=1.
\end{split}
\end{equation}
To continue, we consider
\begin{equation}\label{test}
(\overline{w}_{\varepsilon}^{\beta},\overline{z}_{\varepsilon}^{\beta})=(M\eta_{\varepsilon}^{\beta},M\rho\eta_{\varepsilon}^{\beta}),
\end{equation}
with $\rho>0$ to be determined and $M\gg 1$ a constant such that $\Phi_{\g}^{\beta}(\overline{w}_{\varepsilon}^{\beta},\overline{z}_{\varepsilon}^{\beta})<0$.
Then, under this construction, we define the set of paths
$$\G_\e:=\{g\in C([0,1],\mathcal{X}_{0}^{\beta}(\mathcal{C}_{\Omega})\times \mathcal{X}_{0}^{\beta}(\mathcal{C}_{\Omega}))\,;\, g(0)=(0,0),\ g(1)=(\overline{w}_{\varepsilon}^{\beta},\overline{z}_{\varepsilon}^{\beta})\},$$
and we consider the minimax values
\begin{equation*}
c_\e=\inf_{g\in\G_\e} \max_{t \in [0,1]} \Phi_{\g}^{\beta}(g(t)).
\end{equation*}
Next we prove that, in fact,  $c_{\varepsilon}<c_{\beta}^*$ for $\varepsilon$ small enough.
\begin{lemma}\label{levelb}
Assume $p=2_{\beta}^*-1$. Then, there exists $\varepsilon>0$ small enough such that,
\begin{equation}\label{cotfunctional}
\sup_{t\geq0}\Phi_{\g}^{\beta}(t\overline{w}_{\varepsilon}^{\beta},t\overline{z}_{\varepsilon}^{\beta})<c_{\beta}^*,
\end{equation}
provided that $N>6\beta$.
\end{lemma}
\begin{proof}
Because of \eqref{estimaciones} with $\theta=\beta$, it follows that
\begin{align*}
g(t):=&\Phi_{\g}^{\beta}(t\overline{w}_{\varepsilon}^{\beta},t\overline{z}_{\varepsilon}^{\beta})\\
=&\frac{M^2t^2}{2}\left(\kappa_{\beta}\|\eta_{\varepsilon}^{\beta}\|_{\mathcal{X}_0^{\beta}(\mathcal{C}_{\Omega})}^{2}
+\rho^2 \kappa_{\beta}\|\eta_{\varepsilon}^{\beta}\|_{\mathcal{X}_0^{\beta}(\mathcal{C}_{\Omega})}^{2}
-2\sqrt{\g}\|\sigma_{\varepsilon}^{\theta}\|_{L^2(\Omega)}^{2}\right)-\frac{Mt^{2_{\beta}^*}}{2_{\beta}^*}\\
=&\frac{M^2t^2}{2}\left([\kappa_{\beta}S(\beta,N)+O(\varepsilon^{N-2\beta})]+\rho^2[\kappa_{\beta}S(\beta,N)+O(\varepsilon^{N-2\beta})]
-2\sqrt{\gamma}\|\sigma_{\varepsilon}^{\theta}\|_{L^2(\Omega)}^{2}\right)\\ & -\frac{M^{2_{\beta}^*}t^{2_{\beta}^*}}{2_{\beta}^*}.
\end{align*}
It is clear that $\displaystyle \lim_{t\to \infty} g(t)=-\infty$, therefore, the function $g(t)$ possesses a maximum value at the point
\begin{equation*}
t_{\gamma,\varepsilon}:=\left(\frac{M^2\left([\kappa_{\beta}S(\beta,N)+O(\varepsilon^{N-2\beta})]+\rho^2[\kappa_{\beta}S(\beta,N)+O(\varepsilon^{N-2\beta})]-2\sqrt{\gamma}\|\sigma_{\varepsilon}^{\theta}\|_{L^2(\Omega)}^{2}\right)}{M^{2_{\beta}^*}}\right)^{\frac{1}{2_{\beta}^*-2}}.
\end{equation*}
Moreover, at this point $t_{\gamma,\varepsilon}$,
\begin{equation*}
\begin{array}{rl}
g(t_{\gamma,\varepsilon})=\left(\frac{1}{2}-\frac{1}{2_{\beta}^*}\right)\Big(&\!\!\!\![\kappa_{\beta}S(\beta,N)+O(\varepsilon^{N-2\beta})]+
\rho^2[\kappa_{\beta}S(\beta,N)+O(\varepsilon^{N-2\beta})]\\
&\!\!\!\!-2\sqrt{\gamma}\|\sigma_{\varepsilon}^{\theta}\|_{L^2(\Omega)}^{2}\Big)^{\frac{2_{\beta}^*}{2_{\beta}^*-2}}.
\end{array}\end{equation*}
To finish it is enough to show that
\begin{equation}\label{ppte}
g(t_{\gamma,\varepsilon})<\left(\frac{1}{2}-\frac{1}{2_{\beta}^*}\right)\left(\kappa_{\beta}S(\beta,N)\right)^{\frac{2_{\beta}^*}{2_{\beta}^*-2}}=c_{\beta}^*,
\end{equation}
holds true for $\e$ sufficiently small and making the appropriate choice of $\rho>0$. Thus, sim\-pli\-fying \eqref{ppte}, we are left to choose $\rho>0$ such that
\begin{equation*}
O(\varepsilon^{N-2\beta})+\kappa_{\beta}S(\beta,N)\rho^2+O(\varepsilon^{N-2\beta})\rho^2<2\sqrt{\gamma}\rho \|\sigma_{\varepsilon}^{\theta}\|_{L^2(\Omega)}^{2},
\end{equation*}
holds true provided $\e$ is small enough. To this end, take $\rho=\varepsilon^{\delta}$ with $\delta>0$ to be determined, then, since
\begin{equation*}
O(\varepsilon^{N-2\beta})+\kappa_{\beta}S(\beta,N)\varepsilon^{2\delta}+O(\varepsilon^{N-2\beta+2\delta})=O(\e^{\tau}),
\end{equation*}
with $\tau=\min\{N-2\beta, 2\delta, N-2\beta+2\delta\}=\min\{N-2\beta, 2\delta\}$, the proof will be finished once $\delta>0$ has been chosen such that the inequality
\begin{equation}\label{toprove}
O(\e^{\tau})<2\sqrt{\gamma}\rho\|\sigma_{\varepsilon}^{\theta}\|_{L^2(\Omega)}^{2},
\end{equation}
holds true for $\e$ small enough. Now we use the estimates \eqref{estimaciones}. Then, if $N=4\beta$ inequality \eqref{toprove} reads
\begin{equation}\label{i.1}
O(\e^{\tau})<2C\sqrt{\gamma}\varepsilon^{2\beta+\delta}|\log(\varepsilon)|.
\end{equation}
Since $0<\e\ll 1$, inequality \eqref{i.1} holds for $\tau=\min\{2\beta,2\delta\}>2\beta+\delta$, that is impossible and, thus, inequality \eqref{toprove} can not hold for $N=4\beta$. On the other hand, if $N>4\beta$ inequality \eqref{toprove} has the form,
\begin{equation}\label{i.2}
O(\e^{\tau})<2C\sqrt{\gamma}\varepsilon^{2\beta+\delta}.
\end{equation}
Since $\e\ll 1$, inequality \eqref{i.2} holds for $\tau=\min\{N-2\beta,2\delta\}>2\beta+\delta$. Using the identity $\displaystyle\min\{a,b\}=\frac{1}{2}(a+b-|a-b|)$,
we arrive at the condition
\begin{equation}\label{i.3}
N-2\beta-|N-2\beta-2\delta|>4\beta.
\end{equation}
Finally, we have two options,
\begin{enumerate}
\item  $N-2\beta>2\delta$ combined with \eqref{i.3} provides us with the range,
\begin{equation}\label{i.4}
N-2\beta>2\delta>4\beta.
\end{equation}
Then $N>6\beta$ necessarily, so that we can choose a positive $\delta$ satisfying \eqref{i.4} and, hence, inequality \eqref{toprove} holds for $\e$ small enough.
\item $N-2\beta<2\delta$ combined with \eqref{i.3} implies that $2(N-2\beta)-4\beta>2\delta$, and hence,
\begin{equation}\label{i.5}
2(N-2\beta)-4\beta>2\delta>N-2\beta,
\end{equation}
Once again $N>6\beta$ necessarily, so that we can choose a positive $\delta$ satisfying \eqref{i.5} and, hence, inequality \eqref{toprove} holds for $\e$ small enough.
\end{enumerate}
Thus, if $N>6\beta$ we can choose $\rho>0$ and $\varepsilon>0$ small enough such that \eqref{cotfunctional} is achieved.
\end{proof}

Now, we are in the position to conclude the proof of the second main result of the paper. First we will focus on the particular case when $\alpha=2\beta$. Later on we will follow a similar
argument to prove the results when $\alpha\neq 2\beta$.

\begin{proof}[Proof of Theorem \ref{Th1}. Case $\alpha=2\beta$.]\hfill\break
By Lemma \ref{lezeroextension}, the functional $\Phi_{\g}^{\beta}$ satisfies the MP geometry. Because of MPT we have a PS sequence which by Lemma
\ref{levelb}, satisfies that the corresponding energy level is bellow the critical one.
Taking into account Lemma \ref{PScondition_extensionsistemabb}, this PS sequence satisfies the PS condition, hence we obtain a
critical point  $(w,z)\in\mathcal{X}_0^{\beta}(\mathcal{C}_{\O})\times\mathcal{X}_0^{\beta}(\mathcal{C}_{\O})$ for the functional $\Phi_{\g}^{\beta}$.
The rest of the proof follows as in the subcritical case.
\end{proof}

Now, we focus on the functional $\Phi_{\g}^{\alpha,\beta}$. For this case, we consider
\begin{equation}\label{test2}
(\overline{w}_{\varepsilon}^{\mu},\overline{z}_{\varepsilon}^{\beta})=(M\eta_{\varepsilon}^{\mu},M\rho\eta_{\varepsilon}^{\beta}),
\end{equation}
with $\rho>0$ to be determined and a constant $M\gg 1$ such that $\Phi_{\g}^{\alpha,\beta}(\overline{w}_{\varepsilon}^{\mu},\overline{z}_{\varepsilon}^{\beta})<0$. Let us notice that, by definition,
\begin{equation*}
\sigma_{\varepsilon}^{\mu}\sigma_{\varepsilon}^{\beta}
=\frac{\phi_ru_{\varepsilon}^{\mu}\phi_ru_{\varepsilon}^{\beta}}{\|\phi_ru_{\varepsilon}^{\mu}\|_{2_{\mu}^*}
\|\phi_ru_{\varepsilon}^{\beta}\|_{2_{\beta}^*}},
\end{equation*}
and, since $\mu:=\alpha-\beta$, we find
\begin{equation*}
u_{\varepsilon}^{\mu}u_{\varepsilon}^{\beta}=\frac{\varepsilon^{\frac{N-2\mu}{2}}}{(\varepsilon^2+|x|^2)^{\frac{N-2\mu}{2}}}
\frac{\varepsilon^{\frac{N-2\beta}{2}}}{(\varepsilon^2+|x|^2)^{\frac{N-2\beta}{2}}}=\frac{\varepsilon^{N-\alpha}}{(\varepsilon^2+|x|^2)^{N-\alpha}}
=\left(\frac{\varepsilon^{\frac{N-2(\alpha/2)}{2}}}{(\varepsilon^2+|x|^2)^{\frac{N-2(\alpha/2)}{2}}}\right)^2 \!=\!\left(u_{\varepsilon}^{\alpha/2}\right)^{2}.
\end{equation*}
Thus, applying \eqref{estimaciones} with $\theta=\frac{\alpha}{2}$, we conclude
\begin{equation}\label{estab}
\int_{\Omega}\sigma_{\varepsilon}^{\mu}\sigma_{\varepsilon}^{\beta}dx=C\|\sigma_{\varepsilon}^{\alpha/2}\|_{L^2(\Omega)}^{2}=\left\{
        \begin{tabular}{lc}
        $C \varepsilon^{\alpha}+O(\varepsilon^{N-\alpha})$ & if $N>2\alpha$, \\
                $C \varepsilon^{\alpha}|\log(\varepsilon)|$ & if $N=2\alpha$.
        \end{tabular}
        \right.\\
\end{equation}
Following the steps performed for the case $\alpha=2\beta$, we define the set of paths
$$\G_\e:=\{g\in C([0,1],\mathcal{X}_{0}^{\mu}(\mathcal{C}_{\Omega})\times \mathcal{X}_{0}^{\beta}(\mathcal{C}_{\Omega}))\,;\, g(0)=(0,0),\; g(1)=(M\eta_{\varepsilon}^{\mu},M\rho\eta_{\varepsilon}^{\beta})\},$$
and we consider the minimax values
$$c_\e=\inf_{g\in\G_\e} \max_{t \in [0,1]} \Phi_{\g}^{\alpha,\beta}(g(t)).$$
The final step of our scheme will be completed once we have shown that $c_{\varepsilon}<c_{\mu}^*$ for $\varepsilon$ small enough.
\begin{lemma}\label{levelab}
Assume $p=2_{\beta}^*-1$. Then, there exists $\varepsilon>0$ small enough such that,
\begin{equation}\label{cotfunctionalab}
\sup_{t\geq0}\Phi_{\g}^{\alpha,\beta}(t\overline{w}_{\varepsilon}^{\mu},t\overline{z}_{\varepsilon}^{\beta})<c_{\mu}^*,
\end{equation}
provided that $N>4\alpha-2\beta$.
\end{lemma}
The proof is similar to the one performed for Lemma \ref{levelb}, but we include it for the reader's convenience.
\begin{proof}
Because of \eqref{estimaciones}, it follows that
\begin{align*}
g(t):=&\,\Phi_{\g}^{\alpha,\beta}(t\overline{w}_{\varepsilon}^{\mu},t\overline{z}_{\varepsilon}^{\beta})\\
=&\frac{M^2t^2}{2}\left(\frac{\kappa_{\mu}}{\gamma^{1-\beta/\alpha}}\|\eta_{\varepsilon}^{\mu}\|_{\mathcal{X}_0^{\mu}(\mathcal{C}_{\Omega})}^{2}+\frac{\rho^2\kappa_{\beta}}{\gamma^{\beta/\alpha}}\|\eta_{\varepsilon}^{\beta}\|_{\mathcal{X}_0^{\beta}(\mathcal{C}_{\Omega})}^{2}-2\|\sigma_{\varepsilon}^{\alpha/2}\|_{L^2(\Omega)}^{2}\right)-\frac{M^{2_{\mu}^*}t^{2_{\mu}^*}}{2_{\mu}^*\gamma^{1-\beta/\alpha}}\\
=&\frac{M^2t^2}{2}\left(\frac{1}{\gamma^{1-\beta/\alpha}}[\kappa_{\mu}S(\mu,N)+O(\varepsilon^{N-2\mu})]+\frac{\rho^2}{\gamma^{\beta/\alpha}}[\kappa_{\beta}S(\beta,N)+O(\varepsilon^{N-2\beta})]-2\|\sigma_{\varepsilon}^{\alpha/2}\|_{L^2(\Omega)}^{2}\right)\\
&-\frac{M^{2_{\mu}^*}t^{2_{\mu}^*}}{2_{\mu}^*\gamma^{1-\beta/\alpha}}.
\end{align*}
It is clear that $\displaystyle \lim_{t\to \infty} g(t)=-\infty$,
therefore, the function $g(t)$ possesses a maximum value at the
point, $$
\begin{array}{rl}
t_{\gamma,\varepsilon}\!=\Big(\frac{\gamma^{1-\beta/\alpha}}{M^{2_{\mu}^*-2}}\Big(
& \!\!\!\!\frac{1}{\gamma^{1-\beta/\alpha}}[\kappa_{\mu}S(\mu,N)\!+\!O(\varepsilon^{N-2\mu})]\\
 & \!\! \!\! +
\frac{\rho^2}{\gamma^{\beta/\alpha}}[\kappa_{\beta}S(\beta,N)+O(\varepsilon^{N-2\beta})]-2\|
\sigma_{\varepsilon}^{\alpha/2}\|_{L^2(\Omega)}^{2}\Big)\Big)^{\frac{1}{2_{\mu}^*-2}}.
\end{array}$$
Moreover, at this point $t_{\gamma,\varepsilon}$,
$$\begin{array}{rl}
h(t_{\gamma,\varepsilon})=\left(\frac{1}{2}-\frac{1}{2_{\mu}^*}\right)\Big(\Big(  \gamma^{1-\beta/\alpha}\Big)^{\frac{2}{2_{\mu}^*}}
\Big(&\!\!\!\!\frac{1}{\gamma^{1-\beta/\alpha}}[\kappa_{\mu}S(\mu,N)+O(\varepsilon^{N-2\mu})]\\
& \!\!\!\! +\frac{\rho^2}{\gamma^{\beta/\alpha}}[\kappa_{\beta}S(\beta,N)+O(\varepsilon^{N-2\beta})]-
2\|\sigma_{\varepsilon}^{\alpha/2}\|_{L^2(\Omega)}^{2}\Big)\Big)^{\frac{2_{\mu}^*}{2_{\mu}^*-2}}.
\end{array}$$
To complete the proof we must show that the inequality
\begin{equation}\label{lls}
h(t_{\gamma,\varepsilon})<c_{\mu}^*:=\frac{1}{\gamma^{1-\beta/\alpha}}\left(\frac{1}{2}-\frac{1}{2_{\mu}^*}\right)\left(\kappa_{\mu}S(\mu,N)\right)^{\frac{2_{\mu}^*}{2_{\mu}^*-2}},
\end{equation}
holds true for $\e$ small enough. Thus, simplifying \eqref{lls}, we are left to choose $\rho>0$ such that inequality
\begin{equation*}
O(\varepsilon^{N-2\mu})+\rho^2[\kappa_{\beta}S(\beta,N)+O(\varepsilon^{N-2\beta})]<2\gamma^{\beta/\alpha}\|\sigma_{\varepsilon}^{\alpha/2}\|_{L^2(\Omega)}^{2}.
\end{equation*}
holds true provided $\e$ is small enough. To this end, take $\rho=\e^{\delta}$ with $\delta>0$ to be determined, therefore, since
\begin{equation*}
O(\varepsilon^{N-2\mu})+\kappa_{\beta}S(\beta,N)\e^{2\delta}+O(\varepsilon^{N-2\beta+2\delta})=O(\e^{\tau}),
\end{equation*}
with $\tau=\min\{N-2\mu,2\delta,N-2\beta+2\delta\}=\min\{N-2\mu,2\delta\}$, the proof will be completed once we choose $\delta>0$ such that the inequality
\begin{equation}\label{toproveab}
O(\e^{\tau})<2\gamma^{\beta/\alpha}\|\sigma_{\varepsilon}^{\alpha/2}\|_{L^2(\Omega)}^{2},
\end{equation}
holds true for $\e$ small enough. We use once again the estimates \eqref{estimaciones}. If $N=2\alpha$, because of \eqref{estab}, inequality \eqref{toproveab} reads,
\begin{equation}\label{h.1}
O(\e^{\tau})<2\gamma^{\beta/\alpha}\e^{\alpha+\delta}|\log(\e)|.
\end{equation}
Since $\e\ll 1$, inequality \eqref{h.1} holds for $\tau=\min\{2\alpha-2\mu,2\delta\}=\min\{2\beta,2\delta\}>\alpha+\delta$. Using the identity $\displaystyle\min\{a,b\}=\frac{1}{2}(a+b-|a-b|)$, we find that $\tau>\alpha+\delta$ implies $\beta+\delta-|\beta-\delta|>\alpha+\delta$, which is impossible because $\alpha>\beta$. Therefore, \eqref{toproveab} can not hold if $N=2\alpha$. On the other hand, if $N>2\alpha$, inequality \eqref{toproveab} has the form,
\begin{equation}\label{h.2}
O(\e^{\tau})<2\gamma^{\beta/\alpha}\e^{\alpha+\delta}.
\end{equation}
Since $\e\ll 1$, inequality \eqref{h.2} holds if and only if
$\tau=\min\{N-2\mu,2\delta\}>\alpha+\delta$. Keeping in mind the identity $\displaystyle\min\{a,b\}=\frac{1}{2}(a+b-|a-b|)$, if $\tau>\alpha+\delta$ we arrive at the condition
\begin{equation}\label{h.3}
N-2\mu-|N-2\mu-2\delta|>2\alpha.
\end{equation}
Consequently, we have two options:
\begin{enumerate}
\item $N-2\mu>2\delta$ combined with \eqref{h.3} provides us with the range,
\begin{equation}\label{h.4}
N-2\mu>2\delta>2\alpha.
\end{equation}
Then $N>4\alpha-2\beta$ necessarily, so that we can choose a positive $\delta$ satisfying \eqref{h.4} and, hence, inequality \eqref{toproveab} holds for $\e$ small enough.
\item $N-2\mu<2\delta$ combined with \eqref{h.3} implies that $2(N-2\mu)-2\alpha>2\delta$, and hence,
\begin{equation}\label{h.5}
2(N-2\mu)-2\alpha>2\delta>N-2\mu.
\end{equation}
Once again $N>4\alpha-2\beta$ necessarily, so that we can choose a positive $\delta$ satisfying \eqref{h.5} and, hence, inequality \eqref{toproveab} holds for $\e$ small enough.
\end{enumerate}

\end{proof}
To conclude, we complete the proof of Theorem \ref{Th1}, by dealing with the remaining case $\alpha\neq2\beta$.
\begin{proof}[Proof of Theorem \ref{Th1}. Case $\alpha\neq2\beta$.]\hfill\break
By Lemma \ref{lezeroextension}, the functional $\Phi_{\g}^{\alpha, \beta}$ satisfies the MP geometry. Because of MPT we have a PS sequence which by Lemma
\ref{levelab}, satisfies that the corresponding energy level is bellow the critical one.
Taking into account Lemma \ref{PScondition_extensionsistemaab}, this PS sequence satisfies the PS condition, hence we obtain a
critical point  $(w,z)\in\mathcal{X}_0^{\mu}(\mathcal{C}_{\O})\times\mathcal{X}_0^{\beta}(\mathcal{C}_{\O})$ for the functional $\Phi_{\g}^{\alpha,\beta}$.
The rest of the proof follows as in the subcritical case.
\end{proof}


\begin{thebibliography}{9}

\bibitem{AAP} A. Ambrosetti, J. Garcia Azorero, I. Peral, \textit{Elliptic variational problems in $\mathbb{R}^N$ with critical growth.} Special issue in celebration of Jack K. Hale's 70th birthday, Part 1 (Atlanta, GA/Lisbon, 1998). J. Differential Equations {\bf 168} (2000), no. 1, 10--32.

\bibitem{AR} A. Ambrosetti, P.H. Rabinowitz, {\em Dual variational methods in critical point theory and applications.}
J. Funct. Anal., {\bf 14} (1973), 349--381.

\bibitem{BCdPS} B. Barrios, E. Colorado, A. de Pablo, U. S\'anchez,  \textit{On some critical problems for the fractional Laplacian
Operator}. J. Differential Equations {\bf 252} (2012), no. 11, 6133--6162.

\bibitem{BrCdPS} C. Br\"andle, E. Colorado, A. de Pablo, U. S\'anchez, \textit{A concave-convex elliptic problem involving the fractional Laplacian.}
Proc. Roy. Soc. Edinburgh Sect. A {\bf 143} (2013), no. 1, 39--71.

\bibitem{BN} H. Brezis, L. Nirenberg, \textit{Positive solutions of nonlinear elliptic equations involving critical Sobolev exponents}.
Comm. Pure Appl. Math. \textbf{36} (1983), no. 4, 437--477.

\bibitem{CaSi}X. Cabr\'e, Y. Sire, \emph{Nonlinear equations for fractional Laplacians, I:
Regularity, maximum principles, and Hamiltonian estimates.} Ann.
Inst. H. Poincar\'e Anal. Non Lin\'eaire  {\textbf 31}  (2014),  no. 1, 23--53.

\bibitem{CS} L. Caffarelli, L. Silvestre, \textit{An extension problem related to the fractional Laplacian}. Comm. Partial Differential Equations
\textbf{32} (2007), no. 7-9, 1245--1260.

\bibitem{EV} L.C. Evans, {\em Partial Differential Equations.} Graduate Studies in Mathematics {\bf 19}. American Mathematical Society,
Providence, RI, 1998. xviii+662 pp. ISBN: 0-8218-0772-2.

\bibitem{FKS} E.B. Fabes, C.E. Kenig, R.P. Serapioni, \textit{The local regularity of solutions of degenerate elliptic equations}. Comm. Partial
Differential Equations {\bf 7} (1982), no. 1, 77--116.

\bibitem{Lions} P.-L. Lions, \textit{The concentration-compactness principle in the calculus of variations. The limit case. II.} Rev. Mat.
Iberoamericana {\bf 1} (1985), no. 2, 45--121.

\bibitem{MusNa2} R. Musina, A. I. Nazarov (2014) \textit{On Fractional Laplacians--2}. Ann. Inst. H. Poincar\'e Anal. Non Lin\'eaire {\bf 33}
(2016),  no. 6, 1667--1673.

\bibitem{MusNa3} R. Musina, A. I. Nazarov. \textit{On fractional Laplacians--3.} ESAIM Control Optim. Calc. Var. {\bf 22}  (2016),  no. 3, 832--841.

\end{thebibliography}
\end{document}